\documentclass[11pt, a4paper]{scrartcl}
\pdfoutput=1

\usepackage[utf8]{inputenc}
\usepackage[T1]{fontenc}
\usepackage{lmodern}
\usepackage[a4paper, left=2cm, right=2cm, top=2.5cm, bottom=2.5cm]{geometry}
\usepackage{amsmath}
\usepackage{amsthm}
\usepackage{url}
\usepackage{accents}

\usepackage[square, numbers, comma, sort&compress]{natbib} 
\usepackage{braket,upgreek,bbold}

\usepackage{mathtools}
\usepackage{amssymb}
\usepackage{xcolor}
\usepackage{aligned-overset}
\usepackage{float}

\usepackage{subcaption}
\usepackage{graphicx}
\usepackage[export]{adjustbox}

\usepackage{tikz}
\usepackage{tikz-cd}
\usetikzlibrary{fadings}
\usetikzlibrary{matrix}
\usetikzlibrary{patterns}
\usetikzlibrary{arrows,calc,decorations.pathreplacing,decorations.markings,shapes.geometric,shadows}

\usepackage{esint}

\colorlet{SurfaceColor}{orange!30!white}
\colorlet{BackSurfaceColor}{orange!20!white}
\colorlet{IdentityColor}{red!80!black}
\colorlet{IdentityLineColor}{gray}
\colorlet{IdentityColorMid}{red!50!white}
\colorlet{DefectColor}{blue!50!black}
\colorlet{GreenColor}{green!80!black}
\colorlet{SidelineColor}{black}
\tikzset{
	Surface/.style={ SurfaceColor, opacity=0.8 },
	BackSurface/.style={ BackSurfaceColor, opacity=0.8 },
	LineDefect/.style = {ultra thick, DefectColor },
	Sideline/.style = {	very thin , black },
	IdentitySlice/.style = { very thick , IdentityColor },
	IdentitySliceMid/.style = { very thick , IdentityColorMid },
	Multiplication/.style = {thick, BlueColor },
	IdentityLine/.style = {thin, IdentityLineColor,dashed},
	DefaultSettings/.style = {very thick,scale=0.7,color=blue!50!black, baseline=0.5cm}
}
\tikzset{
    string/.style={draw=#1, postaction={decorate}, decoration={markings,mark=at position .51 with {\arrow[draw=#1]{>}}}},
	costring/.style={draw=#1, postaction={decorate}, decoration={markings,mark=at position .51 with {\arrow[draw=#1]{<}}}},
	ostring/.style={draw=#1, postaction={decorate}, decoration={markings,mark=at position .47 with {\arrow[draw=#1]{>}}}},
	ustring/.style={draw=#1, postaction={decorate}, decoration={markings,mark=at position .56 with {\arrow[draw=#1]{>}}}},
	oostring/.style={draw=#1, postaction={decorate}, decoration={markings,mark=at position .43 with {\arrow[draw=#1]{>}}}},
	uustring/.style={draw=#1, postaction={decorate}, decoration={markings,mark=at position .59 with {\arrow[draw=#1]{>}}}},
	directed/.style={string=blue!50!black}, 
	odirected/.style={ostring=blue!50!black}, 
	udirected/.style={ustring=blue!50!black}, 
	oodirected/.style={oostring=blue!50!black}, 
	uudirected/.style={uustring=blue!50!black},     
	redirected/.style={costring= blue!50!black},
	redirectedgreen/.style={costring= green!50!black},
	directedgreen/.style={string= green!50!black},
	redirectedlightgreen/.style={costring= green!65!black},
	directedlightgreen/.style={string= green!65!black},
	redirectedred/.style={costring= red!50!black},
	directedred/.style={string= red!50!black},
}

\tikzset{-dot-/.style={decoration={
			markings,
			mark=at position 0.5 with {\fill circle (2pt);}},postaction={decorate}}}

\tikzset{
	Fdot/.style={circle, draw, fill, inner sep=0pt}, 
	Odot/.style={circle, draw, inner sep=0.1pt, minimum size=0.1cm}
}

\newcommand\tikzzbox[1]
{#1}

\definecolor{Myblue}{rgb}{0,0,0.6}  
\usepackage[colorlinks,allcolors=Myblue,breaklinks=true]{hyperref}

\usepackage[shortcuts]{extdash}

\newtheorem{theorem}{Theorem}[section]
\newtheorem{lemma}[theorem]{Lemma}

\newtheorem{proposition}[theorem]{Proposition}

\theoremstyle{definition}
\newtheorem{definition}[theorem]{Definition}

\newtheorem{remark}[theorem]{Remark}

\DeclareMathOperator{\Hom}{Hom}
\DeclareMathOperator{\im}{im}

\def\ev{\mathrm{ev}}
\def\coev{\mathrm{coev}}
\def\id{\mathrm{id}}
\def\Id{\mathrm{Id}}
\def\End{\mathrm{End}}
\def\Aut{\mathrm{Aut}}
\def\op{\mathrm{op}}
\def\1{{\mathbf1}}

\newcommand{\C}{{\mathcal{C}}}

\newcommand{\B}{{\mathcal{B}}}

\newcommand{\A}{{\mathbb{A}}}

\newcommand{\vphi}{{\varphi}}

\newcommand{\dm}[1]{d_{#1}}

\newcommand\isomto{\stackrel{\sim}{\smash{\longrightarrow}\rule{0pt}{0.4ex}}}

\newcommand{\iso}{\cong}

\renewcommand{\to}{\longrightarrow}
\renewcommand{\mapsto}{\longmapsto}

\numberwithin{figure}{section}

\newcommand{\pic}[2][0.75]{
	\begin{tikzpicture}[scale=0.5,baseline={([yshift=-.5ex]current bounding box.center)}]
	\node at (0,0) {\includegraphics[scale=#1]{pics/#2}};
	\end{tikzpicture}
}

\newcommand{\eqrefO}[1]{\hyperref[fig:O#1]{\text{(O#1)}}}
\newcommand{\eqrefT}[1]{\hyperref[fig:T#1]{\text{(T#1)}}}

\title{\boldmath Generalised Orbifolds and $G$-equivariantisation}

\author{Sebastian Heinrich$^\dagger$, Julia Plavnik$^{\ddagger *}$, Ingo Runkel$^\dagger$, Abigail Watkins$^\ddagger$
\\[0.5cm]
\normalsize\slshape $^\dagger$ Fachbereich Mathematik, Universität Hamburg, Germany
\\
\normalsize{\texttt{sebastian.heinrich, ingo.runkel@uni-hamburg.de}}
\\[0.5cm]
\normalsize\slshape $^\ddagger$ Department of Mathematics, Indiana University, Bloomington, USA
\\
\normalsize{\texttt{jplavnik, abwatk@iu.edu}}
\\[0.5cm]
\normalsize\slshape $^*$ Department of Mathematics and Data Science, Vrije Universiteit Brussel, Brussels, Belgium
\\}

\date{}

\begin{document}

\maketitle
\begin{abstract}
    In a construction motivated by topological field theory, a so-called orbifold datum $\A$ in a ribbon category $\C$ allows one to define a new ribbon category $\C_\A$. If $\C$ is the neutral component of a $G$-crossed ribbon category $\hat\B$, and $\A$ is an orbifold datum in $\C$ defined in terms of $\hat\B$, one finds that $\C_\A$ is equivalent to the equivariantisation $\hat\B^G$ of $\hat\B$ as a ribbon category. We give a constructive proof of this equivalence.
\end{abstract}

\vspace{4em}

\tableofcontents

\newpage

\section{Introduction}

Topological field theory has proven to be an inspiring bridge linking algebra and topology. One example of this is the so-called orbifold construction which we briefly describe below. On the algebraic side, it allows one to define new ribbon categories in terms of given ones. In this paper we show that a standard construction in the context of ribbon categories -- the equivariantisation of a $G$-crossed ribbon category -- is an instance of the orbifold construction. 

This relation has been conjectured before, and can in fact be shown relatively quickly via an abstract argument, which we sketch below. However, the main point of this paper is to give a constructive proof in a more general setting. Along the way we collect all the relevant definitions, which may in itself be useful as they are otherwise scattered across several papers using different conventions.

\medskip

Let us briefly introduce the players in this paper and state our main theorem. 

\medskip

We start with $G$-crossed ribbon categories and $G$-equivariantisation, see \cite{Turaev:2000ug,Kirillov:2001uz,Mueger:2002}, as well as \cite[Ch.\,8.24]{EGNO-book} and \cite[Ch.\,VI]{TuraevBook2010}.
Let $G$ be a finite group. A $G$-crossed (additive) ribbon category $\hat\B$ is, firstly, a $G$-graded monoidal category, $\hat\B = \bigoplus_{g \in G} \B_g$, whose tensor product is multiplicative with respects to the grading. It is also equipped with a right $G$-action which moves components by conjugation: the action of $h$ is a functor $\varphi(h)$ from $\B_g$ to $\B_{h^{-1}gh}$. Note that for $G$ non-commutative, $\hat\B$ cannot be braided as $X \otimes Y$ and $Y \otimes X$ would in general be in different graded components. Instead one introduces a $G$-crossed braiding, where, for $X \in \B_g$ and $Y \in B_h$ one has $c_{X,Y} : X \otimes Y \to Y \otimes \varphi(h)(X)$. The target object indeed has the same grade $hh^{-1}gh=gh$ as the source object. An analogous modification applies to the ribbon twist. The definitions are such that the neutral component $\B_e$ is an actual ribbon category.

A $G$-equivariant object is an object $X \in \hat\B$ together with a coherent family of isomorphisms $\eta_g : \varphi(g)(X) \to X$, $g \in G$. One finds that the category $\hat\B^G$ formed by these objects is a ribbon category. This is our first key player, and we review the definition in more detail in Section~\ref{sec:crossed_equiv}.

\medskip

Next we turn to the orbifold construction, or internal state sum construction. This was originially introduced in two-dimensional conformal field theory \cite{Frohlich:2009gb} and then formulated for topological field theories in arbitrary dimension \cite{Carqueville:2017aoe}. In a nutshell, it describes data assigned to a stratification which allows one to define a new TFT in terms of a given one. In the case of three-dimensional surgery TFT, the orbifold construction is developed in detail in \cite{Carqueville:2017ono,Carqueville:2018sld,Carqueville:2021edn}. The algebraic counterpart one finds is as follows (after placing it in an appropriately general context): Let $\C$ be an idempotent-complete additive ribbon category. An orbifold datum $\A$ in $\C$ consists of
\begin{itemize}
    \item a suitable algebra object $A \in \C$,
    \item an $A$-$(A\otimes A)$-bimodule $T$ in $\C$, 
    \item several additional morphisms,
\end{itemize}
all subject to a list of conditions. Part of these conditions state that $T$ defines a tensor product on $A$-mod, the category of $A$-modules in $\C$. An orbifold datum $\A$ in $\C$ allows one to define a new ribbon category $\C_\A$, the category of line defects \cite{Mulevicius:2022,Carqueville:2021edn}. We review this construction in Section~\ref{sec:orbdata}.

Written out explicitly, the data and conditions for $\A$ and $\C_\A$ may seem baroque, but they have very natural interpretations in TFT \cite{Carqueville:2017aoe,Carqueville:2021dbv}. On the algebraic side, they unify three standard constructions \cite{Carqueville:2018sld}:
\begin{itemize}
    \item Given a suitable commutative algebra $A$ in $\C$, the category $\C_A^\mathrm{loc}$ of local modules is again a ribbon category. From $A$ one can construct an orbifold datum $\A$ such that $\C_A^\mathrm{loc} \cong \C_\A$ as ribbon categories. 

    \item Take $\C$ to be the category of finite-dimensional complex vector spaces. Given a spherical fusion category $\mathcal{S}$, its Drinfeld centre $\mathcal{Z}(\mathcal{S})$ is a ribbon category. From $\mathcal{S}$ one can construct an orbifold datum $\A$ in $\C = \mathrm{vect}$, such that $\mathcal{Z}(\mathcal{S}) \cong \C_\A$ as ribbon categories.

    \item Take $\C = \B_e =: \B$ to be the neutral component of a $G$-crossed ribbon category $\hat\B$. From $\hat\B$ one can construct an orbifold datum $\A$ in $\B$ such that the category of line defects agrees with the equivariantisation, $\hat\B^G \cong \B_\A$  as ribbon categories. Here, $A = \bigoplus_{g \in G} A_g$ is such that $A_g$-mod is equivalent to the graded component $\B_g$, and $T$ encodes the tensor product of $\hat\B$. We review this construction in more detail in Section~\ref{sec:G-crossed-to-orb}.
\end{itemize}
The ribbon equivalence in the last point is only conjectured in \cite[Rem.\,5.5]{Carqueville:2018sld} in a more restricted setting, and our main result is
(up to the existence of certain square roots which could be absorbed into changing conventions):

\medskip

\noindent
\textbf{Theorem.} (Theorem~\ref{thm:main})
\textit{Let $G$ be a finite group and $\hat\B$ an additive idempotent-complete $G$-crossed ribbon category. Suppose that  $|G|$ is invertible and 
that each graded component $\B_g$ contains an object with invertible quantum dimension. Then from $\hat\B$ one can construct an orbifold datum $\A$ in $\B = \B_e$ such that $\B_\A \cong \hat{\B}^G$ as ribbon categories.}

\medskip

We prove this by providing, in Section~\ref{sec:equiv}, an explicit ribbon equivalence $E : \hat{\B}^G \to \B_\A$.

\medskip

To conclude the introduction, let us sketch how to obtain an abstract proof of the above theorem, which however needs a more restrictive setting. Namely, let $\hat\B$ be a $G$-crossed ribbon fusion category such that $\B = \B_e$ is modular. By \cite[Thm.\,2]{Cui:2015cbf}, $\B \boxtimes (\hat\B^G)^\text{rev} \cong \mathcal{Z}(\hat\B)$ (this is stated in \cite{Cui:2015cbf} as a braided equivalence, and under the additional assumption of unitarity).
On the other hand, by \cite[Prop.\,7.4]{Mulevicius:2022gce} we have 
$\B \boxtimes (\B_\A)^\text{rev} \cong \mathcal{Z}(\B_\A^2)$ as ribbon categories for a suitable $\B_\A^2$ defined in \cite{Mulevicius:2022gce}. One now needs to show that $\B_\A^2 \cong \hat\B$ as fusion categories. It is straightforward to get a functor $\hat\B \cong A\text{-mod} \to \B_\A^2$, and it remains
to show that it is a monoidal equivalence compatible with the half braidings for $\B$.
Note that this proof strategy does not give an explicit functor between $\hat{\B}^G$ and $\B_\A$.

\subsection*{Acknowledgements}

The authors thank 
    Nils Carqueville, 
    Vincentas Mulevi\v cius, 
    and
    Sean Sanford 
    for helpful discussions.
We also thank Nils Carqueville and Benjamin Haake for coordinating the arXiv submission of the present paper with their paper \cite{Carqueville-Haake}, where -- amongst other results -- they independently prove Theorem~\ref{thm:main} via the same explicit equivalence. 

JP and AW are grateful for the hospitality and excellent working conditions at the Department of Mathematics at the Universit\"at Hamburg, where this project started and part of it was completed. 
SH thanks the Department of Mathematics of Indiana University for hospitality in summer 2024, where part of the project was completed. 

SH and IR thank  the  Deutsche Forschungs\-gemeinschaft (DFG, German Research Foundation) under Germany's Excellence Strategy - EXC 2121 ``Quantum Universe'' - 390833306,  and the Collaborative Research Center - SFB 1624 ``Higher structures, moduli spaces and integrability'' - 506632645, for support.
JP and AW thank the NSF under grant DMS-2146392 and the Simons Foundation under the Award 889000 as part of the Simons Collaboration on Global Categorical Symmetries for support. JP thanks the Alexander von Humboldt Foundation since part of this work was completed as a Humboldt Experienced Fellow at Universit\"at Hamburg.

\subsection*{Conventions}

Throughout this paper, $G$ is a finite group, and $\hat\B$ and $\C$ are idempotent-complete additive categories. Functors between additive categories are always taken to be additive.
We do not require the categories to be linear over a field.

In addition $\C$ is ribbon, and $\hat\B$ is $G$-crossed ribbon (we recall our conventions in Section~\ref{sec:crossed_equiv}). 
The graded components of $\hat\B$ are denoted by $\hat{\B} = \bigoplus_{g \in G} \B_g$ with neutral component $\B := \B_e$.
We assume the monoidal structure of $\C$ and $\hat\B$ to be strict to simplify notation.

All string diagrams are written using the optimistic convention and are read bottom to top and left to right. 

\section{\texorpdfstring{\boldmath $G$}{G}-crossed ribbon categories and equivariantisation}\label{sec:crossed_equiv}

In this section, we quickly recall the notion of a $G$-crossed ribbon category $\hat\B$ and its $G$-equivariantisation $\hat\B^G$. For more details we refer to \cite[Ch.\,8.24]{EGNO-book} (which defines $G$-crossed braided categories with coherences, but does not define the $G$-crossed ribbon structure) and \cite[Ch.\,VI.2]{TuraevBook2010} (which includes the ribbon case but assumes the $G$-action to be strict). The definition in \cite[Sec.\,3]{Galindo:2024} is closest to what we use here, as there the underlying category is strict, but the action is not. In Section~\ref{sec:equiv}, we will assume that the $G$-action is also strict, but until then we will keep track of the associated natural isomorphisms. Note that throughout this paper we use right actions while these sources use left actions.

\medskip

Let $\hat\B = \bigoplus_{g\in G}\B_g$ be an additive monoidal category which is faithfully graded, i.e.\ $\B_g\not= \emptyset$ for all $g\in G$, and let $\Aut_\otimes(\hat\B)$ be the monoidal category of its (additive and) monoidal autoequivalences and monoidal natural transformations.

Let $\underline{G^{op}}$ denote the discrete monoidal category whose objects are the group elements of $G$ with tensor product given by reversing the group multiplication so that $g \otimes h = hg$.
A \textit{right $G$-action} on a monoidal category $\C$ is a monoidal functor 
\[
\varphi\colon  \underline{G^\op}\to \Aut_\otimes(\C). 
\]
We will denote the associated natural transformations by 
\[
\mu_g^{X,Y}:= \vphi(g)(X \otimes Y) \isomto \vphi(g)(X) \otimes \vphi(g) (Y)
\quad \text{and} \quad
\gamma_{g,h}^X: \vphi(g)\vphi(h) (X) \isomto \vphi(hg)(X)  \ ,
\]
and refer to them as \textit{tensorators} and \textit{compositors}, respectively.
Without loss of generality we may assume that 
\[
\varphi(e)=\Id ,
\]
rather than the two just being naturally isomorphic. 
The isomorphism for the tensor unit is written as $\mu_g^0 : \varphi(g)(\1) \isomto \1$.

\begin{remark}\label{rem:G-strict}
Without loss of generality, we can and will assume that $\hat\B$ is strict as a monoidal category. By \cite[Thm.\,1.1]{Galindo:2016}, we may also assume that the right action is \textit{strict} in the sense that the natural isomorphism $\mu_g$ and $\gamma_{g,h}$ above are identities. We will make use of this result in  Section~\ref{sec:equiv}, but for now we keep these isomorphisms explicit.
\end{remark}

\begin{definition}
    A \textit{(right) $G$-crossed ribbon} category is a faithfully graded additive monoidal category $\hat\B=\bigoplus_{g\in G}\B_g$ with:
    \begin{itemize}
        \item A right action $\varphi\colon  \underline{G^\op}\to \Aut_\otimes(\hat\B)$ satisfying $\varphi(h)(\B_g)\subset \B_{h^{-1}gh}$, for all $g,h\in G$.
        \item A \textit{$G$-braiding} consisting of natural isomorphisms $c_{X,Y_h}\colon X\otimes Y_h\to Y_h\otimes \varphi(h)(X)$, for all $X\in \hat\B$, $Y_h\in \B_h$, satisfying:
\begin{enumerate}
    \item (Compatibility with the $G$-action)  
    \[\adjustbox{scale = 0.85, center}{\begin{tikzcd}
        {\varphi(g)(X \otimes Y_h)} && {\varphi(g)(Y_h \otimes \varphi(h)(X))} && {\varphi(g)(Y_h) \otimes \varphi(g)\varphi(h)(X)} \\
    	\\
        {\varphi(g)(X) \otimes \varphi(g)(Y_h)} && {\varphi(g)(Y_h) \otimes \varphi(g^{-1}hg)\varphi(g)(X)} && {\varphi(g)(Y_h) \otimes \varphi(hg)(X)}
        \arrow["{\varphi(g)c_{X,Y_h}}", from=1-1, to=1-3]
        \arrow["{\mu_{g}^{X,Y_h}}"', from=1-1, to=3-1]
        \arrow["{\mu_g^{Y_h,\varphi(h)(X)}}", from=1-3, to=1-5]
        \arrow["{\id_{\varphi(g)(Y_h)}\otimes \gamma^X_{g,h}}", from=1-5, to=3-5]
        \arrow["{c_{\varphi(g)(X), \varphi(g)(Y_h)}}"', from=3-1, to=3-3]
        \arrow["{\id_{\varphi(g)(Y)} \otimes \gamma^X_{g^{-1}hg,g}}"', from=3-3, to=3-5]
    \end{tikzcd}}\]

    \item (Twisted hexagon identities) Fo all $Z_k \in \B_k$,
    \[\begin{tikzcd}
	{X \otimes Y_h \otimes Z_k} &&& {Z_k \otimes \varphi(k)(X \otimes Y_h)} \\
	\\
	{X \otimes Z_k \otimes \varphi(k) (Y_h)} &&& {Z_k \otimes \varphi(k)(X) \otimes \varphi(k)(Y_h)}
	\arrow["{c_{X \otimes Y_h, Z_k}}", from=1-1, to=1-4]
	\arrow["{\id_X \otimes c_{Y_h,Z_k}}"', from=1-1, to=3-1]
	\arrow["{\id_{Z_k}\otimes\mu_k^{X,Y_h}}", from=1-4, to=3-4]
	\arrow["{c_{X,Z_k} \otimes \id_{\varphi(k)(Y_h)}}"', from=3-1, to=3-4]
    \end{tikzcd}\]
    \[\begin{tikzcd}
	{X \otimes Y_h \otimes Z_k} &&& {Y_h \otimes Z_k \otimes \varphi(hk)(X)} \\
	\\
	{Y_h \otimes \varphi(h)(X) \otimes Z_k} &&& {Y_h  \otimes Z_k \otimes \varphi(k) \varphi(h)(X) }
	\arrow["{c_{X, Y_h \otimes Z_k}}", from=1-1, to=1-4]
	\arrow["{c_{X,Y_h} \otimes \id_{Z_k}}"', from=1-1, to=3-1]
	\arrow["{\id_{{Y_h} \otimes Z_k} \otimes (\gamma^X_{k,h})^{-1}}", from=1-4, to=3-4]
	\arrow["{\id_{Y_h} \otimes c_{\varphi(h)(X),Z_k} }"', from=3-1, to=3-4]
    \end{tikzcd}\]
\end{enumerate}
\item A \textit{$G$-ribbon twist} consisting of a natural isomorphism $\theta_{Y_h} : Y_h \to \varphi(h)(Y_h)$ such that $\theta_{\1}= \id_{\1} $ and satisfying:
\begin{enumerate}
    \item (Compatibility with the $G$-braiding)
    \[\begin{tikzcd}
	{X_g \otimes Y_h} &&&& \varphi(gh)(X_g \otimes Y_h) \\
	\\
	{Y_h \otimes \varphi(h)(X_g)} &&&& {\varphi(gh)(X_g)  \otimes \varphi(gh)(Y_h) }
    \\
	\\
	{ \varphi(h)(X_g) \otimes \varphi(h^{-1}gh) (Y_h)} &&&& { \varphi(h)\varphi(g) (X_g) \otimes \varphi(h^{-1}gh) \varphi(h)(Y_h)}
	\arrow["{\theta_{Xg\otimes Y_h }}", from=1-1, to=1-5]
	\arrow["{c_{X_g,Y_h}}"', from=1-1, to=3-1]
	\arrow["{ \mu^{X_g,Y_h}_{gh}}", from=1-5, to=3-5]
    \arrow["{c_{Y_h,\varphi(h)(X_g)}}"', from=3-1, to=5-1]
	\arrow["{ (\gamma^{X_g}_{h,g})^{-1}\otimes(\gamma^{Y_h}_{h^{-1}gh,h})^{-1}}", from=3-5, to=5-5]
	\arrow["{\varphi(h)(\theta_{X_g})\otimes \varphi(h^{-1}gh)(\theta_{Y_h}) }"', from=5-1, to=5-5]
    \end{tikzcd}\]
    \item (Compatibility with the $G$-action)
\[\begin{tikzcd}
	{\varphi(g)(Y_h)} && {\varphi(g)\varphi(h)(Y_h)} \\
	\\	{\varphi(g^{-1}hg)\varphi(g)(Y_h)} && {\varphi(hg)(Y_h)}
	\arrow["{\varphi(g)(\theta_{Y_h})}", from=1-1, to=1-3]
	\arrow["{\theta_{\varphi(g)({Y_h})}}"', from=1-1, to=3-1]
\arrow["{\gamma^{Y_h}_{g,h}}", from=1-3, to=3-3]
	\arrow["{\gamma^{Y_h}_{g^{-1}hg,g}}"', from=3-1, to=3-3]
    \end{tikzcd}\]
    \item (Compatibility with duals)
    \[\adjustbox{scale = 0.95, pagecenter}{\begin{tikzcd}
	{\1} &
    \varphi(g^{-1})(\1) 
    && \varphi(g^{-1})\big(X_g \otimes X_g^*\big) &&& {\varphi(g^{-1})(\varphi(g)(X_g)\otimes X_g^*)} \\
	\\
   X_g \otimes X_g^* &&& {X_g  \otimes \varphi(g^{-1})(X_g^*) } &&& {\varphi(g^{-1})\varphi(g)(X_g)\otimes \varphi(g^{-1})(X_g^*)}
	\arrow["{(\mu_{g^{-1}}^0)^{-1}}", from=1-1, to=1-2]
	\arrow["{\varphi(g^{-1})\big(\coev_{X_g}\big)}", from=1-2, to=1-4]
	\arrow["{\coev_{X_g}}"', from=1-1, to=3-1]
	\arrow["{\varphi(g^{-1})(\theta_{X_g}\otimes \id_{X^*_g})}", from=1-4, to=1-7]
	\arrow["{\id_{X_g} \otimes \theta_{X^*_g} }"', from=3-1, to=3-4]
    \arrow["{(\gamma_{g^{-1},g}^{X_g})^{-1}\otimes\id_{\varphi(g^{-1})(X_g^*)}}"', from=3-4, to=3-7]
    \arrow["{\mu^{\varphi(g)(X_g), X^*_g}_{\varphi(g^{-1})}}", from=1-7, to=3-7]
    \end{tikzcd}}\]
\end{enumerate}        
    \end{itemize}
\end{definition}

We can adapt the graphical calculus for braided monoidal categories to $G$-crossed braided categories. However, we have to carefully keep track of the labels of the strings, since every time we braid, the upper strand catches an action of the component of the lower strand, as seen in
\[c_{X,Y_h} =
\begin{tikzpicture}[very thick,scale=0.5,color=blue!50!black, baseline]
\draw[color=blue!50!black] (-1,-1) node[below] (A1) {{\scriptsize$X$}};
\draw[color=blue!50!black] (1,-1) node[below] (A2) {{\scriptsize$Y_h$}};
\draw[color=blue!50!black] (-1,1) node[above] (B1) {{\scriptsize$Y_h\vphantom{\varphi}$}};
\draw[color=blue!50!black] (1,1) node[above] (B2) {{\scriptsize$\varphi(h)(X)$}};
\draw[color=blue!50!black] (A2) -- (B1);
\draw[color=white, line width=4pt] (A1) -- (B2);
\draw[color=blue!50!black] (A1) -- (B2);
\end{tikzpicture} 
\colon X\otimes Y_h \to Y_h \otimes \varphi(h)(X).\]
For the opposite crossing we will keep the convention that the out-going object of the string that crosses over gets acted upon,
\[\tilde c_{Y_h,X} =
\begin{tikzpicture}[very thick,scale=0.5,color=blue!50!black, baseline]
\draw[color=blue!50!black] (-1,-1) node[below] (A1) {{\scriptsize$Y_h$}};
\draw[color=blue!50!black] (1,-1) node[below] (A2) {{\scriptsize$X$}};
\draw[color=blue!50!black] (-1,1) node[above] (B1) {{\scriptsize$\varphi(h^{-1})(X)$}};
\draw[color=blue!50!black] (1,1) node[above] (B2) {{\scriptsize$Y_h\vphantom{\varphi}$}};
\draw[color=blue!50!black] (A1) -- (B2);
\draw[color=white, line width=4pt] (A2) -- (B1);
\draw[color=blue!50!black] (A2) -- (B1);
\end{tikzpicture} 
\colon Y_h\otimes X \to \varphi(h^{-1})(X) \otimes Y_h.\]
We define $\tilde c_{Y_h,X}$ in terms of the inverse braiding as
\[
\tilde c_{Y_h,X} =
\Big[ 
Y_h\otimes X
\xrightarrow{\id \otimes (\gamma_{h,h^{-1}}^{X})^{-1}}
Y_h\otimes \varphi(h)\varphi(h^{-1})(X)
\xrightarrow{c_{X,Y_h}^{-1}}
\varphi(h^{-1})(X) \otimes Y_h \Big].
\]

\medskip

Next we recall the notion of $G$-equivariantisation, see \cite[Sec. 4.15]{EGNO-book} and \cite[App.\,5]{TuraevBook2010}.

\begin{definition}
    Let $\mathcal{M}$ be a strict monoidal category with $G$-action $\varphi$. The \textit{$G$-equivariantisation} $\mathcal{M}^G$ of $\mathcal{M}$ is the monoidal category in which:
\begin{itemize}
    \item 
    Objects are pairs $\mathbf{X} = (X,(\eta^X_g)_{g\in G})$, where $X\in \mathcal{M}$ and $\eta^X_g\colon \varphi(g)(X)\isomto X$ belongs to a family of isomorphisms in $\mathcal{M}$ satisfying
    \[
    \begin{tikzcd}
    \varphi(g)\varphi(h)(X) \arrow[d, "\gamma_{g,h}"'] \arrow[r, "\varphi(g)(\eta^X_h)"] & \varphi(g)(X) \arrow[d, "\eta^X_g"] \\
    \varphi(hg)(X) \arrow[r, "\eta^X_{hg}"'] & X 
    \end{tikzcd}
    \]
    for all $h\in G$.
\item
    Morphisms $f\colon \mathbf{X} \to \mathbf{Y}$ are morphisms $f\colon X\to Y$ in $\mathcal{M}$ for which
    \[\begin{tikzcd}
    \varphi(g)(X) \arrow[r, "\eta^X_g"] \arrow[d, "\varphi(g)(f)"'] & X \arrow[d, "f"] \\
    \varphi(g)(Y) \arrow[r, "\eta^Y_g"'] & Y               
    \end{tikzcd}
    \]
    commutes for all $g\in G$.
\item
    The monoidal product is given by
    \[ \mathbf{X} \otimes \mathbf{Y} \coloneqq \left(X\otimes Y, \left[\varphi(g)(X\otimes Y)\xrightarrow[]{\mu_g^{X,Y}}\varphi(g)(X)\otimes \varphi(g)(Y)\xrightarrow[]{\eta^X_g\,\otimes\,\eta^Y_g}X\otimes Y\right]_{g\in G}\right).\]
    As $\mathcal{M}$ is strict, one checks that this product is again strictly associative.
\item  
    The monoidal unit in $\mathcal{M}^G$ is $\1 \coloneqq (\1,\eta^\1)$ with $\eta^\1_g: \varphi(g)(\1) \to \1$ given by the structure morphism $\mu_g^0$ of the monoidal structure of $\varphi(g) : \mathcal{M} \to \mathcal{M}$.
\end{itemize}
\end{definition}

We remark that since $\varphi(g)(X_h) \in \B_{g^{-1}hg}$, in terms of components $X_h$ of $X$ we have
\[
\eta_g^X\big|_{\varphi(g)(X_h)} : \varphi(g)(X_h) \isomto X_{g^{-1}hg} \ .
\]

The $G$-equivariantisation of a $G$-crossed ribbon category is a ribbon category, see \cite[Ch.\,8.24]{EGNO-book} (for the braided case) and \cite[Sec.\,3.4]{Galindo:2024} (for the ribbon case).

\begin{proposition}
    Let $\hat\B$ be a $G$-crossed ribbon category. Then $\hat\B^G$ is a ribbon category with braiding $c^{\hat\B^G}_{\mathbf X,\mathbf Y} = \bigoplus_{h\in G} c^{\hat\B^G}_{X,Y_h}$, 
    and ribbon twist  $\theta^{\hat\B^G}_{\mathbf Y} = \bigoplus_{h\in G} \theta^{\hat\B^G}_{Y_h}$, 
    where
\begin{align*}    
    c^{\hat\B^G}_{X,Y_h} &\coloneqq \Big[ X\otimes Y_h \xrightarrow[]{c^{\hat\B}_{X,Y_h}} Y_h\otimes \varphi(h)(X)\xrightarrow[]{\id\,\otimes\,\eta^X_h} Y_h\otimes X \Big] \ ,
    \\
    \theta^{\hat\B^G}_{Y_h} &\coloneqq \Big[ Y_h \xrightarrow[]{\theta^{\hat\B}_{Y_h}} \varphi(h)(Y_h)\xrightarrow[]{\eta^{Y_{h}}_h} Y_h \Big] \ .
\end{align*}    

\end{proposition}

\section{Orbifold data and the category of line defects}\label{sec:orbdata}
In this section we review the definition of orbifold data in a ribbon category from \cite{Carqueville:2018sld} and the construction of the corresponding category of line defects as introduced in \cite{Mulevicius:2022}.

We will need a little bit of background on Frobenius algebras and relative tensor products. For the definition of a symmetric Frobenius algebra in the present conventions, we refer to \cite[Sec.\,2.2]{Mulevicius:2022}. 
Here we just note that $\Delta$-separability of $A$ means that the coproduct $\Delta$ is a separability idempotent for $A$, that is, $\mu \circ \Delta = \id_A$. This implies that the relative tensor product $\otimes_A$ can be written as the image of an idempotent, which is the reason that we insist on idempotent-complete categories. 

In more detail, consider a right $A$-module $M$ (action $\rho_M$) and a left $A$-module $N$ (action $\rho_N$). 
The relative tensor product $M \otimes_A N$ is given by the image of the idempotent 
\[
P_{M,N} = \big[ M \otimes N \xrightarrow{\id \otimes (\Delta \circ \eta) \otimes \id} M \otimes A \otimes A \otimes N \xrightarrow{\rho_M \otimes \rho_N} M \otimes N \big] \ ,
\]
see e.g.\ \cite[Lem.\,1.21]{Kirillov:2001ti} or \cite[App.\,A]{Barmeier:2007idm}. Since $\C$ is idempotent-complete, we obtain projection and embedding maps
\[
\pi_{M,N} : M \otimes N \to M \otimes_A N
~~,\quad
\iota_{M,N} : M \otimes_A N \to M \otimes N
\]
such that
\[
\pi_{M,N} \circ \iota_{M,N} = \id_{M \otimes_A N}
~~,\quad
\iota_{M,N} \circ \pi_{M,N} = P_{M,N}.
\]
Because of this, we will often use the same notation for a map $M \otimes_A N \to M' \otimes_A N'$ and for the corresponding map $M \otimes N \to M' \otimes N'$ obtained via the embedding/projection maps, especially in string diagrams.

The next definition involves modules over tensor products of algebras. The convention for turning the tensor product of two algebras $A,B \in \C$ into an algebra is
\[
A \otimes B \otimes A \otimes B \xrightarrow{\id \otimes c_{B,A} \otimes \id} A \otimes A \otimes B \otimes B
\xrightarrow{\mu_A \otimes \mu_B} A \otimes B ,
\]
rather than using the inverse braiding $(c_{A,B})^{-1}$.

\begin{figure}[p!]
	\captionsetup[subfigure]{labelformat=empty}
	\centering
	\begin{subfigure}[b]{0.5\textwidth}
		\centering
		\pic[1.5]{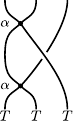}$=$\pic[1.5]{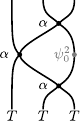}
		\caption{}
		\label{fig:O1}
	\end{subfigure}\hspace{-2em}\raisebox{5.5em}{(O1)}\\
	\begin{subfigure}[b]{0.4\textwidth}
		\centering
		\pic[1.5]{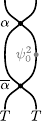}$=$\pic[1.5]{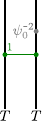}
		\caption{}
		\label{fig:O2}
	\end{subfigure}\hspace{-2em}\raisebox{5.5em}{(O2)}
	\begin{subfigure}[b]{0.4\textwidth}
		\centering
		\pic[1.5]{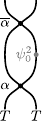}$=$\pic[1.5]{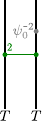}
		\caption{}
		\label{fig:O3}
	\end{subfigure}\hspace{-2em}\raisebox{5.5em}{(O3)}\\
	\begin{subfigure}[b]{0.4\textwidth}
		\centering
		\pic[1.5]{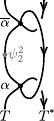}$=$\pic[1.5]{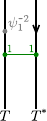}
		\caption{}
		\label{fig:O4}
	\end{subfigure}\hspace{-2em}\raisebox{5.5em}{(O4)}
	\begin{subfigure}[b]{0.4\textwidth}
		\centering
		\pic[1.5]{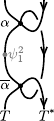}$=$\pic[1.5]{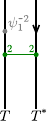}
		\caption{}
		\label{fig:O5}
	\end{subfigure}\hspace{-2em}\raisebox{5.5em}{(O5)}\\
	\begin{subfigure}[b]{0.4\textwidth}
		\centering
		\pic[1.5]{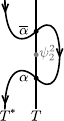}$=$\pic[1.5]{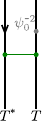}
		\caption{}
		\label{fig:O6}
	\end{subfigure}\hspace{-2em}\raisebox{5.5em}{(O6)}
	\begin{subfigure}[b]{0.4\textwidth}
		\centering
		\pic[1.5]{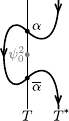}$=$\pic[1.5]{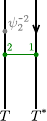}
		\caption{}
		\label{fig:O7}
	\end{subfigure}\hspace{-2em}\raisebox{5.5em}{(O7)}\\
	\begin{subfigure}[b]{0.8\textwidth}
		\centering
		\pic[1.5]{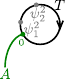}$=$\pic[1.5]{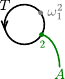}$=$\pic[1.5]{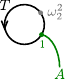}$=$\pic[1.5]{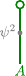}$\cdot \phi^{-1}$
		\caption{}
		\label{fig:O8}
	\end{subfigure}\hspace{-2em}\raisebox{4.25em}{(O8)}

\vspace*{-1em}
\caption{Identities an orbifold datum $\A$ has to satisfy. 
We use the convention explained before Definition~\ref{def:orb-datum}, i.e.\ we use the name $\alpha$ for the map $T \otimes T \to T \otimes T$ obtained from $\alpha\colon T\otimes_2 T\to T\otimes_1 T$ via projection/embedding maps, etc.} 
\label{fig:orb_ids}
\end{figure}

\begin{definition}[{\cite[Def.\,3.4]{Carqueville:2018sld}}]\label{def:orb-datum}
Let $\C$ be an idempotent-complete additive ribbon category. 
An orbifold datum $\A$ in $\C$ is a tuple $\A=(A,T,\alpha,\overline{\alpha},\psi,\phi)$, where
\begin{itemize}
    \item $A$ is a $\Delta$-separable symmetric Frobenius algebra in $\C$,
    \item $T$ is an $A$-$(A\otimes A)$-bimodule in $\C$, and we indicate left action by a subindex of $0$, and the right actions of the first and second
tensor factor by subindexes $1$, and $2$,
    \item $\alpha\colon T\otimes_2 T\to T\otimes_1 T$ is an $A$-$AAA$-bimodule morphism (where we stopped writing out all tensor product symbols),
    \item $\overline{\alpha}\colon T\otimes_1 T\to T\otimes_2 T$ is an $A$-$AAA$-bimodule morphism,
    \item $\psi\in \End_{AA}(A)$ is an invertible
 $A$-$A$-bimodule morphism, and
    \item $\phi$ is an invertible element in $\End(\1)$.\footnote{We follow the conventions in \cite[Def.\,2.2]{Mulevicius:2022} and not that of \cite[Def.\,3.4]{Carqueville:2018sld}. Namely, \eqrefO{8} involves $\phi^{-1}$ and not $\phi^{-2}$.}
\end{itemize}
These data have to satisfy the axioms \eqrefO{1}--\eqrefO{8} 
in Figure \ref{fig:orb_ids} where the following notation is used:
\begin{equation*}
\pic[1.5]{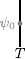}   := \pic[1.5]{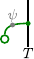}, 	\quad
\pic[1.5]{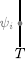}   := \pic[1.5]{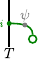}, 	\quad
\pic[1.5]{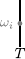} := \pic[1.5]{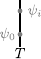},   \quad
i = 1,2.
\end{equation*}
\end{definition}

\begin{definition}[{\cite[Def.\,3.1]{Mulevicius:2022}}]\label{def:cat-of-wilson}
Fix an orbifold datum $\A$ in $\C$. 
The \textit{category of line defects} 
$\C_\A$ is defined by:
\begin{itemize}
    \item \textit{Objects}: An object is a tuple $(M, \tau_1, \tau_2, \overline{\tau_1}, \overline{\tau_2})$ consisting of an 
    $A$-$A$-bimodule $M$, and 
    \begin{itemize}
        \item $A$-$AAA$-bimodule morphisms $\tau_1: M \otimes_0 T \to T \otimes_1 M$, and $\tau_2: M \otimes_0 T \to T \otimes_2 M$, and
        \item $A$-$AAA$-bimodule morphisms $\overline{\tau_1}: T \otimes_1 M \to M \otimes_0 T$, and $\overline{\tau_2}: T \otimes_2 M \to M \otimes_0 T$, represented by
    \end{itemize} 
\begin{equation*}
\tau_i = \pic[1.5]{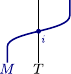}, \qquad
\overline{\tau_i} := \pic[1.5]{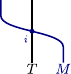}, \qquad i = 1,2 ~,
\end{equation*}
which satisfy the axioms  \eqrefT{1}--\eqrefT{7}
in Figure \ref{fig:wilsonlines_ids}.
    
\begin{figure}[t]
\captionsetup[subfigure]{labelformat=empty}
\centering
 \begin{subfigure}[b]{0.45\textwidth}
\centering
\pic[1.5]{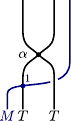}$=$\pic[1.5]{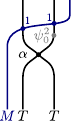}
\caption{}
\label{fig:T1}
\end{subfigure}\hspace{-2em}\raisebox{5.5em}{(T1)}\\
\begin{subfigure}[b]{0.45\textwidth}
\centering
\pic[1.5]{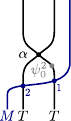}$=$\pic[1.5]{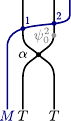}
\caption{}
\label{fig:T2}
\end{subfigure}\hspace{-2em}\raisebox{5.5em}{(T2)}
\begin{subfigure}[b]{0.45\textwidth}
\centering
\pic[1.5]{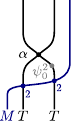}$=$\pic[1.5]{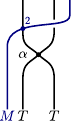}
\caption{}
\label{fig:T3}
\end{subfigure}\hspace{-2em}\raisebox{5.5em}{(T3)}\\
\begin{subfigure}[b]{0.35\textwidth}
\centering
\pic[1.5]{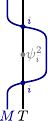}$=$\pic[1.5]{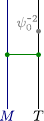}
\caption{}
\label{fig:T4}
\end{subfigure}\hspace{-1em}\raisebox{5.5em}{(T4)}
\begin{subfigure}[b]{0.35\textwidth}
\centering
\pic[1.5]{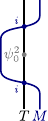}$=$\pic[1.5]{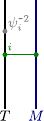}
\caption{}
\label{fig:T5}
\end{subfigure}\hspace{-1em}\raisebox{5.5em}{(T5)}\hspace{4em}\raisebox{5.5em}{$i=1,2$}\\
\begin{subfigure}[b]{0.35\textwidth}
\centering
\pic[1.5]{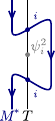}$=$\pic[1.5]{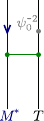}
\caption{}
\label{fig:T6}
\end{subfigure}\hspace{-1em}\raisebox{5.5em}{(T6)}
\begin{subfigure}[b]{0.35\textwidth}
\centering
\pic[1.5]{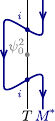}$=$\pic[1.5]{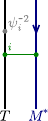}
\caption{}
\label{fig:T7}
\end{subfigure}\hspace{-1em}\raisebox{5.5em}{(T7)}\hspace{4em}\raisebox{5.5em}{$i=1,2$}
	\caption{Identities in the category of line defects.}
    \label{fig:wilsonlines_ids}
 \end{figure}

    \item \textit{Morphisms}: A morphism $f: (M, \tau_1^M, \tau_2^M, \overline{\tau_1}^M, \overline{\tau_2}^M) \to (N, \tau_1^N, \tau_2^N, \overline{\tau_1}^N, \overline{\tau_2}^N)$ is a bimodule morphism $f: M  \to N$ which commutes with the morphisms $\tau_i$, i.e. $\tau_i^N \circ(f \otimes_0 \id_T) = (\id_T \otimes_i f) \circ \tau_i^M$ 
    \begin{equation*}
\pic[1.5]{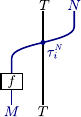} = \pic[1.5]{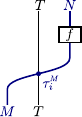}, \quad i=1,2 ~.
\end{equation*}
\end{itemize}
We will refer to the morphisms $\tau_i$ as \textit{T-crossings} and the morphisms $\overline{\tau_i}$ as their \textit{pseudo-inverses}.
\end{definition}

Next we review how to turn $\C_\A$ into a ribbon category \cite[Sec.\,3.2]{Mulevicius:2022}. 
For this we extend the notation for $\psi_i$ and $\omega_i$ to bimodules as follows:
\begin{equation*}
\pic[1.5]{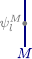} := \pic[1.5]{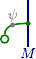}, \quad
\pic[1.5]{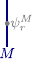} := \pic[1.5]{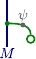}, \quad
\pic[1.5]{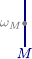} := \pic[1.5]{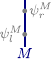}.
\end{equation*}
In practice, we will usually drop the extra index $M$ as it is clear from the context.

The monoidal, rigid, braided, and ribbon structures on $\C_\A$ are: 
\begin{itemize}
    \item \textit{Tensor Product}: Let $$(M, \tau_1^M, \tau_2^M, \overline{\tau_1}^M, \overline{\tau_2}^M) \otimes (N, \tau_1^N, \tau_2^N, \overline{\tau_1}^N, \overline{\tau_2}^N) := (M \otimes N, \tau_1^{M \otimes N}, \tau_2^{M \otimes N}, \overline{\tau_1}^{M \otimes N}, \overline{\tau_2}^{M \otimes N})$$ where $\tau_i^{M \otimes N}$ and $\overline{\tau_i}^{M \otimes N}$ are given by

    \begin{equation*}\label{eq:tau^MN_i}
\tau^{M\otimes N}_i            := \pic[1.5]{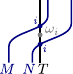}, \quad
\overline{\tau_i}^{M\otimes N} := \pic[1.5]{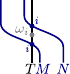}, \qquad i=1,2 ~;
\end{equation*}

    \item \textit{Monoidal Unit}: The monoidal unit $\1_{\C_\A}:= (A, \tau_1^A, \tau_2^A, \overline{\tau_1}^A, \overline{\tau_2}^A)$ with $\tau_i^{A}$ and $\overline{\tau_i}^{A}$ given by: 
    \begin{equation*}
\label{eq:unit_T-cross}
\pic[1.5]{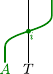} := \pic[1.5]{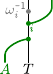}, \quad
\pic[1.5]{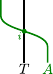} := \pic[1.5]{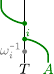}, \qquad i=1,2 ~;
\end{equation*}

    \item \textit{Associator and Unitors}: The associators and unitors of this tensor product are inherited from the category $_A \C_A$ of $A$-$A$-bimodules in $\C$;

    \item \textit{Duals}: For any object $(M, \tau_1^M, \tau_2^M, \overline{\tau_1}^M, \overline{\tau_2}^M) $ in $\C_\A$, take its dual in $\C_\A$ to be the dual bimodule $M^*$ with the T-crossings and pseudo-inverses given below:  
    \begin{equation*}
\pic[1.5]{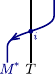}
:= \pic[1.5]{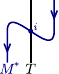}, \quad
\pic[1.5]{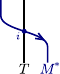}
:= \pic[1.5]{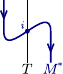} \qquad i=1,2
\end{equation*}
with evaluation and coevaluation maps given by
\begin{align*}
&[\ev_M: M^* \otimes_A M \to A] := \pic[1.5]{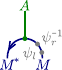},
&&[\coev_M: A \to M \otimes_A M^*] := \pic[1.5]{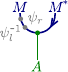}, \nonumber \\
\label{eq:evs_evts_in_D} 
&[\widetilde{ev}_M: M \otimes_A M^* \to A] := \pic[1.5]{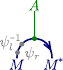},
&&[\widetilde{coev}_M: A \to M^* \otimes_A M] := \pic[1.5]{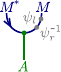};
\end{align*}

    \item \textit{Braiding}: The braiding of two objects $(M, \tau_1^M, \tau_2^M, \overline{\tau_1}^M, \overline{\tau_2}^M)$ and $(N, \tau_1^N, \tau_2^N, \overline{\tau_1}^N, \overline{\tau_2}^N)$ in $\C_\A$ is given by the isomorphism $c_{M,N}$: 
    \begin{equation*}
\label{eq:D_braiding}
c_{M,N} := \pic[1.5]{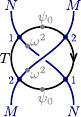} \cdot \phi
\ , \qquad
c^{-1}_{M,N} := \pic[1.5]{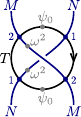} \cdot \phi
\ ;
\end{equation*}

    \item \textit{Ribbon Twist}: For $(M, \tau_1^M, \tau_2^M, \overline{\tau_1}^M, \overline{\tau_2}^M) \in \C_\A$, we define the twist $\theta_M$ using the braiding on $\C_\A$ in the typical way. 
\begin{equation*}
\label{eq:twist}
\theta_M 
:=
\pic[1.5]{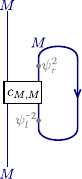} 
=
\pic[1.5]{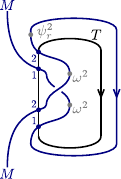} \cdot \phi.
\end{equation*}
\end{itemize}

\begin{proposition}
    $\C_\A$ with the above structure morphisms is an idempotent-complete additive ribbon category.
\end{proposition}

This is shown in \cite[Sec.\,3.2]{Mulevicius:2022} and \cite[Sec.\,4.2]{Carqueville:2021dbv}. The setting in \cite{Mulevicius:2022} is different (there, $\C$ is a modular fusion category over an algebraically closed field), but the proof given there that $\C_\A$ is braided just use the diagrammatic rules and still works in the present setting (the proof that $\C_\A$ is ribbon given there, however, relies on semisimplicity). That $\C_\A$ is ribbon follows along the same lines as in \cite[Sec.\,4.2]{Carqueville:2021dbv} (where instead the starting point is a certain 3-category obtained from a defect TFT). Idempotent-completeness can be checked using the same argument as in the proof of \cite[Prop.\,3.13]{Mulevicius:2022}.

\section{Orbifold data from \texorpdfstring{\boldmath $G$}{G}-crossed ribbon categories}\label{sec:G-crossed-to-orb}

Given a ribbon $G$-crossed category $\hat{\B}$, in \cite[Sec.\,5]{Carqueville:2018sld} a construction of an orbifold datum $\A$ in the trivial component $\B\coloneqq \B_e$ of $\hat{\B}$ is given. Namely, for each component of the grading on $\hat{\B}$, we choose a simple object $m_g \in \B_g$ such that $m_e = \1$ and we set 
$$\dm{g}\coloneqq \dim m_g.$$ 
We need to assume that
\begin{itemize}
    \item for each $g\in G$ there is an invertible element $\sqrt{d_g} \in \End(\1)$ which squares to $d_g$,\footnote{It is possible to adapt the conventions such that no square roots of $d_g$ are needed in the construction. We will not pursue this here and instead stick to the present conventions to match the notation in \cite{Carqueville:2018sld,Mulevicius:2022}. \label{fn:square-root}}
    and
    \item the element $|G| \cdot \id \in \End(\1)$ is invertible.
\end{itemize}

In the string diagrams that follow, we simply label the objects $m_g$ by $g$. We construct an orbifold datum $\A=(A,T,\alpha,\psi,\phi)$ in $\B$ with:
\begin{enumerate}
    \item[(i)]
    the $\Delta$-separable symmetric Frobenius algebra $A$ in $\B = \B_e$ defined as
    
$$A \coloneqq \bigoplus_{g\,\in\,G} A_g, \qquad A_g \coloneqq m_g^* \otimes m_g,$$

with (co)multiplication and (co)unit  $\mu=\bigoplus_g \mu_g$, $\eta=\bigoplus_g\eta_g$, $\Delta=\bigoplus_g \Delta_g$, and $\varepsilon=\bigoplus_g \varepsilon_g$, where the components are given by 
\begin{equation*}
\mu_g\coloneqq
\tikzzbox{%
\begin{tikzpicture}[very thick,scale=1,color=green!50!black, baseline=.9cm]
\draw[line width=0pt] 
(3,0) node[line width=0pt] (D) {{\scriptsize$g^*$}}
(2,0) node[line width=0pt] (s) {{\scriptsize$g\vphantom{g^*}$}}; 
\draw[redirectedgreen] (D) .. controls +(0,1) and +(0,1) .. (s);
\draw[line width=0pt] 
(3.45,0) node[line width=0pt] (re) {{\scriptsize$g\vphantom{g^*}$}}
(1.55,0) node[line width=0pt] (li) {{\scriptsize$g^*$}}; 
\draw[line width=0pt] 
(2.7,2) node[line width=0pt] (ore) {{\scriptsize$g\vphantom{g^*}$}}
(2.3,2) node[line width=0pt] (oli) {{\scriptsize$g^*$}}; 
\draw (li) .. controls +(0,0.75) and +(0,-0.25) .. (2.3,1.25);
\draw (2.3,1.25) -- (oli);
\draw (re) .. controls +(0,0.75) and +(0,-0.25) .. (2.7,1.25);
\draw (2.7,1.25) -- (ore);
\end{tikzpicture}
}%
\!\! , \quad \eta_g\coloneqq
\tikzzbox{%
\begin{tikzpicture}[very thick,scale=1,color=green!50!black, baseline=-.4cm,rotate=180]
\draw[line width=0pt] 
(3,0) node[line width=0pt] (D) {{\scriptsize$g^*$}}
(2,0) node[line width=0pt] (s) {{\scriptsize$g\vphantom{g^*}$}}; 
\draw[redirectedgreen] (s) .. controls +(0,1) and +(0,1)  .. (D);
\end{tikzpicture}
}%
\!\!, \quad 
\Delta_g\coloneqq\frac{1}{\dm{g}}\cdot\!\!
\tikzzbox{
\begin{tikzpicture}[very thick,scale=1,color=green!50!black, baseline=-0.9cm, rotate=180]
\draw[line width=0pt] 
(3,0) node[line width=0pt] (D) {{\scriptsize${g}\vphantom{g^*}$}}
(2,0) node[line width=0pt] (s) {{\scriptsize$g^*$}}; 
\draw[redirectedgreen] (D) .. controls +(0,1) and +(0,1) .. (s);
\draw[line width=0pt] 
(3.45,0) node[line width=0pt] (re) {{\scriptsize$g^*$}}
(1.55,0) node[line width=0pt] (li) {{\scriptsize${g}\vphantom{g^*}$}}; 
\draw[line width=0pt] 
(2.7,2) node[line width=0pt] (ore) {{\scriptsize$g^*$}}
(2.3,2) node[line width=0pt] (oli) {{\scriptsize$g\vphantom{g^*}$}}; 
\draw (li) .. controls +(0,0.75) and +(0,-0.25) .. (2.3,1.25);
\draw (2.3,1.25) -- (oli);
\draw (re) .. controls +(0,0.75) and +(0,-0.25) .. (2.7,1.25);
\draw (2.7,1.25) -- (ore);
\end{tikzpicture}
}%
\!\!, \quad
\varepsilon_g\coloneqq \dm{g}\cdot\!\!
\tikzzbox{
\begin{tikzpicture}[very thick,scale=1,color=green!50!black, baseline=.4cm]
\draw[line width=0pt] 
(3,0) node[line width=0pt] (D) {{\scriptsize$g\vphantom{g^*}$}}
(2,0) node[line width=0pt] (s) {{\scriptsize$g^*$}}; 
\draw[redirectedgreen] (s) .. controls +(0,1) and +(0,1) .. (D);
\end{tikzpicture}
}%
\!\!,
\end{equation*}
\item[(ii)]
the $A$-$(A\otimes A)$-bimodule $T$ defined as
$$T \coloneqq \bigoplus_{g,h\,\in\,G} T_{g,h}, \qquad T_{g,h} \coloneqq m_{gh}^* \otimes m_g \otimes m_h,$$
with actions
\begin{align*}
\begin{tikzpicture}[very thick,scale=0.75,color=blue!50!black, baseline]
\draw (0,-1) node[below] (X) {{\scriptsize$T_{g,h}$}};
\draw[color=green!50!black] (-0.8,-1) node[below] (A1) {{\scriptsize$A_{gh}$}};
\draw (0,1) node[right] (Xu) {};
\draw[color=green!50!black] (A1) .. controls +(0,0.5) and +(-0.5,-0.5) .. (0,0.3);
\draw (0,-1) -- (0,1); 
\fill[color=blue!50!black] (0,0.3) circle (2.9pt) node (meet2) {};
\end{tikzpicture} 
&:=
\pic[3.5]{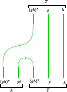}
, 
\\[1em]
\begin{tikzpicture}[very thick,scale=0.75,color=blue!50!black, baseline]
\draw (0,-1) node[below] (X) {{\scriptsize$T_{g,h}$}};
\draw[color=green!50!black] (0.8,-1) node[below] (A1) {{\scriptsize$A_h$}};
\draw (0,1) node[right] (Xu) {};
\draw[color=green!50!black] (A1) .. controls +(0,0.5) and +(0.5,-0.5) .. (0,0.3);
\draw (0,-1) -- (0,1); 
\fill[color=blue!50!black] (0,0.3) circle (2.9pt) node (meet2) {};
\fill[color=black] (0.2,0.5) circle (0pt) node (meet) {{\tiny$2$}};
\end{tikzpicture} 
&:= 
\pic[3.5]{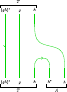},  
\hspace{1cm}
\begin{tikzpicture}[very thick,scale=0.75,color=blue!50!black, baseline]
\draw (0,-1) node[below] (X) {{\scriptsize$T_{g,h}$}};
\draw[color=green!50!black] (0.8,-1) node[below] (A1) {{\scriptsize$A_g$}};
\draw (0,1) node[right] (Xu) {};
\draw[color=green!50!black] (A1) .. controls +(0,0.5) and +(0.5,-0.5) .. (0,0.3);
\draw (0,-1) -- (0,1); 
\fill[color=blue!50!black] (0,0.3) circle (2.9pt) node (meet2) {};
\fill[color=black] (0.2,0.5) circle (0pt) node (meet) {{\tiny$1$}};
\end{tikzpicture} 
:=
\pic[3.5]{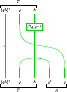} 
.
\end{align*}

\item[(iii)]
the $A$-$AAA$-bimodule morphisms $\alpha\colon T\otimes_2 T \to T\otimes_1 T$ and $\overline{\alpha}\colon T\otimes_1 T \to T\otimes_2 T$ defined as
$$\alpha\coloneqq \bigoplus_{g,h,k\,\in\,G} \alpha_{g,h,k}, \qquad \alpha_{g,h,k}\colon T_{g,hk} \otimes T_{h,k} \to T_{gh, k} \otimes T_{g,h},$$
$$\overline{\alpha}\coloneqq \bigoplus_{g,h,k\,\in\,G} \overline{\alpha}_{g,h,k}, \qquad \overline{\alpha}_{g,h,k}\colon T_{gh,k} \otimes T_{g,h} \to T_{g,hk} \otimes T_{h,k},$$
where the components are given by
\[\alpha_{g,h,k} ~=\hspace{-1em}\includegraphics[scale=4,valign=c]{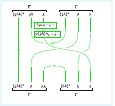} \hspace{-2em}, \quad
\overline{\alpha}_{g,h,k} ~=\hspace{-1em}\includegraphics[scale=4,valign=c]{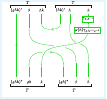}
\]

\item[(iv)]
$$\psi\vert_{A_g} \coloneqq \frac{1}{d_g^{1/2}}\id_{A_g},$$ 
\item[(v)]
$$\phi \coloneqq \frac{1}{|G|}.$$
\end{enumerate}
As an example of how to read the labelling of string diagrams in $G$-crossed ribbon categories, let us look at the diagram for $\alpha_{g,h,k}$ in more detail. When the strand starting at $m_k$ crosses over $m_h$, $m_g$ via $\tilde c_{-,m_k}$, and over $m_{gh}$ via $c_{m_k,m_{gh}}$, its label changes from $m_k$ first to $\varphi(h^{-1})(m_k)$ and then to $\varphi(g^{-1})\varphi(h^{-1})(m_k)$ and $\varphi(gh)\varphi(g^{-1})\varphi(h^{-1})(m_k)$. The morphism $\varphi(gh)(\gamma_{g^{-1},h^{-1}})$ maps this object to $\varphi(gh)\varphi(h^{-1}g^{-1})(m_k)$, and  $\gamma_{gh,h^{-1}g^{-1}}$ to $\varphi(h^{-1}g^{-1}gh)(m_k)= \varphi(e)(m_k)$. Since we assumed $\varphi(e)=\id$, the outgoing part of the strand is indeed labelled $m_k$.

\begin{theorem}[{\cite[Thm.\,5.1\,\&\,Rem.\,5.6]{Carqueville:2018sld}}]
\label{thm:Gcrossed-to-orb}
$\A := (A, T, \alpha, \overline{\alpha}, \psi, \phi)$ defined in 
(i)--(v) is an orbifold datum in $\B_e$.
\end{theorem}

\section{Equivalence}\label{sec:equiv}

Recall that $\B = \B_e$ denotes the neutral component of $\hat\B$.
In this section, we prove the main theorem of this paper:

\begin{theorem}\label{thm:main}
Let $G$ be a finite group and $\hat\B$ an idempotent-complete additive ribbon $G$-crossed category such that${}^{\,\ref{fn:square-root}}$
\begin{itemize}
    \item $|G|$ is invertible in $\End_{\B}(\1)$, and
    \item for each $g \in G$ there is $m_g \in \B_g$ and an invertible $\sqrt{d_g} \in \End_{\B}(\1)$ such that $\dim m_g = (\sqrt{d_g})^2$.
\end{itemize}
Let $\A$ the orbifold datum 
in $\B$ obtained from $\hat{\B}$ via Theorem~\ref{thm:Gcrossed-to-orb}.
Then there is an explicit ribbon equivalence $E \colon \hat{\B}^G \to \B_\A$.
\end{theorem}

We will define the functor $E\colon \hat{\B}^G \to \B_\A$ and prove its properties in four steps:
\begin{enumerate}
    \item Denote by ${}_A \B_A$ the category of $A$-$A$-bimodules in $\B$. We define a faithful monoidal functor $J\colon \hat{\B} \to {}_A\B_A$, and we characterise its image in the Hom-spaces of ${}_A \B_A$.
    \item The functor $J$ can be lifted to a functor $E \colon \hat{\B}^G \to \B_\A$ in the sense that the following diagram commutes (on the nose):
\[
\begin{tikzcd}
    \hat{\B}^G \arrow[d,"\text{forget}"'] \arrow[r,"E"] & \B_\A \arrow[d,"\text{forget}"] \\
    \hat\B \arrow[r,"J"]  & {}_A \B_A 
\end{tikzcd}
\]
    \item We check that $E$ is a ribbon functor.
    \item We check that $E$ is an equivalence.
\end{enumerate}

\subsection[The monoidal functor \texorpdfstring{$J$}{J}]{The monoidal functor \texorpdfstring{\boldmath $J$}{J}}\label{sec:functorJ}

Recall from Section~\ref{sec:G-crossed-to-orb} that the algebra $A$ in the orbifold datum $\A$ decomposes as $A = \bigoplus_{g \in G} A_g$. All $A_g$ are objects in the neutral component $\B = \B_e$. This is not to be confused with the decomposition $X = \bigoplus_{g \in G} X_g$ of an object in $\hat{\B}$ into homogeneous components $X_g \in \B_g$.

Below, we will often use a special case of the relative tensor product over $A$, namely when $M = X \otimes m_g$ and $N = m_h^* \otimes Y$, for objects $X,Y \in \hat{\B}$. The $A$-action is as in the description of $T$ in Section~\ref{sec:G-crossed-to-orb} above. For example, only the component $A_g$ has a non-zero right action on $M$, and it is given by
\[
\rho_M = \big[ M \otimes A_g = X \otimes m_g \otimes m_g^* \otimes m_g \xrightarrow{\id \otimes \widetilde{\ev}_{m_g} \otimes \id} X \otimes m_g = M \big] \ .
\]
We have:

\begin{lemma}\label{lem:tensor-A}
    For $M = X \otimes m_g$ and $N = m_h^* \otimes Y$, the image of $P_{M,N}$ is $0$ if $g \neq h$ and $X \otimes Y$ if $g=h$, that is, $M \otimes_A N = \delta_{g,h} \, X \otimes Y$.
\end{lemma}

The proof is straightforward, and we just state the form of the projector $P_{M,N}$ from Section~\ref{sec:orbdata} in this case,
\[
P_{M,N} = \frac{\delta_{g,h}}{\dm{g}} \Big[ X \otimes m_g \otimes m_g^* \otimes Y 
\xrightarrow{\id \otimes \widetilde{\ev}_{m_g} \otimes \id} X \otimes Y \xrightarrow{\id \otimes \coev_{m_g} \otimes \id} X \otimes m_g \otimes m_g^* \otimes Y \Big] \ .
\]
For the projection and embedding map we choose the convention to include the factor $d_g^{-1}$ with the embedding,
\begin{align*}
\pi_{M,N} 
&= \delta_{g,h} \Big[ X \otimes m_g \otimes m_g^* \otimes Y 
\xrightarrow{\id \otimes \widetilde{\ev}_{m_g} \otimes \id} X \otimes Y \Big] \ ,
\\
\iota_{M,N}
&=
\frac{1}{\dm{g}} \Big[  X \otimes Y \xrightarrow{\id \otimes \coev_{m_g} \otimes \id} X \otimes m_g \otimes m_g^* \otimes Y \Big]
\ .
\end{align*}

Let ${}_A \B_A$ be the category of $A$-$A$-bimodules in $\B$ and define the functor 
\begin{align*}
J\colon \hat{\B} &\to {}_A\B_A,\\
X=\bigoplus_{g\in G} X_g &\mapsto \bigoplus_{g,h\,\in\,G} m_g^* \otimes X_{gh^{-1}} \otimes m_h,\\
\left[\bigoplus_{g\in G} X_g \xrightarrow{\underset{g\,\in\,G}{\bigoplus} f_g} \bigoplus_{g\in G} Y_g\right] &\mapsto \bigoplus_{g,h \in G} \id_{m_g^*} \otimes f_{gh^{-1}} \otimes \id_{m_h},
\end{align*}
where $X_{gh^{-1}}$ is the summand of $X$ in the $gh^{-1}$ component.
The left and right action of $A = \bigoplus_{g \in G} A_g$  on $J(X)$ is given by 
\[
\includegraphics[scale=4.50,valign=c]{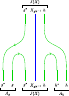}
\qquad .
\]
In more detail, the left action by $A_g$ and right action by $A_h$ projects to the $m_g^* \otimes X_{gh^{-1}} \otimes m_h$ summand of $J(X)$ and is zero on the other summands. Note that, as in Section~\ref{sec:G-crossed-to-orb}, in string diagrams we abbreviate object labels $m_g \equiv g$ to make the diagram less crowded.

\begin{lemma}
The functor $J$ defined above is a faithful monoidal functor with monoidal structure $J^2$ and $J^0$ given by identity morphisms.
\end{lemma}

\begin{proof}
It is immediate that $J$ is faithful as $\id\otimes - \otimes \id$ is injective on $\Hom$-spaces. 
For the monoidal structure, we have
\begin{align*}
    J(X) \otimes_A J(Y) &= \bigoplus_{g,h,k,l\,\in\,G} \Bigl(m_g^* \otimes X_{gk^{-1}} \otimes m_k\Bigr) \otimes_A  \Bigl(m_l^* \otimes  Y_{lh^{-1}} \otimes m_h\Bigr)
    \\
    &= \bigoplus_{g,h,k\,\in\,G} m_g^* \otimes X_{gk^{-1}} \otimes  Y_{kh^{-1}} \otimes m_h 
     \ ,
\end{align*}
where the second equality follows from Lemma~\ref{lem:tensor-A}.
Setting $u = hk^{-1}$, one observes that
$$
\bigoplus_{k\,\in\,G} X_{gk^{-1}} \otimes Y_{kh^{-1}} = \bigoplus_{u\,\in\,G} X_{gh^{-1}u} \otimes Y_{u^{-1}} = (X \otimes Y)_{gh^{-1}} \ ,
$$
so that altogether
$$J(X) \otimes_A J(Y) = \bigoplus_{g,h\,\in\,G} m_g^* \otimes (X \otimes Y)_{gh^{-1}} \otimes m_h = J(X \otimes Y) \ .$$
This is  consistent with the choice $J^2_{X,Y} = \id$, and also with the hexagon as both, $J(X) \otimes_A (J(Y) \otimes_A J(Z))$ and $(J(X) \otimes_A J(Y)) \otimes_A J(Z)$ are equal to $\bigoplus_{g,h\,\in\,G} m_g^* \otimes (X \otimes Y \otimes Z)_{gh^{-1}} \otimes m_h$, and the associator is trivial
(see e.g.\ \cite[App.\,A]{Barmeier:2007idm} where the associator for $\otimes_A$ in the case of $\Delta$-separable Frobenius algebras is discussed in more detail).

For the tensor unit $\1 \in \hat{\B}$ one finds 
$$
J(\1) 
= 
\bigoplus_{g,h\,\in\,G} m_g^* \otimes \1_{gh^{-1}} \otimes m_h
= 
\bigoplus_{g\,\in\,G} m_g^*  \otimes m_g
= 
A \ ,
$$
where we used that $\1_{gh^{-1}} = \delta_{g,h}\, \1$. It is immediate that $J^0=\id$ solves the triangle conditions.
\end{proof}

The next lemma describes the image of $J : \Hom_{\hat{\B}}(X,Y) \to \Hom_{{}_A \B_A}(J(X),J(Y))$.

\begin{lemma}\label{lem:image-J-Hom}
    The image of the functor $J$ defined above on $\Hom$-spaces is
    \[
    \im\left(J_{|\Hom_{\hat{\B}}(X,Y)}\right)=\left\{\begin{array}{l}
        \underbrace{\left[\underbrace{\bigoplus_{g,h\,\in\,G} m_g^* \otimes X_{gh^{-1}} \otimes m_h}_{J(X)}\xrightarrow{\underset{g,h,k,l\,\in\,G}{\bigoplus} f_{g,h,k,l}}\underbrace{\bigoplus_{k,l\,\in\,G} m_k^* \otimes Y_{kl^{-1}} \otimes m_l}_{J(Y)}\right]}_{\in\,\Hom_{{}_A\B_A}(J(X),J(Y))}\\\hline
        \forall g,h,k,l\in G\colon\left\{\begin{array}{rl}
            (i)&g\not= k\Rightarrow f_{g,h,k,l}=0,\\
            (ii)&h\not= l\Rightarrow f_{g,h,k,l}=0,\\
            (iii)&f_{g,h,g,h}=\id_{m_g^*}\otimes f_{g,h}\otimes \id_{m_h},\\
            (iv)&f_{g,e}=f_{gh,h}
            \end{array}\right.
        \end{array}\right\} \ ,
    \]
where  $f_{g,h,k,l} : m_g^* \otimes X_{gh^{-1}} \otimes m_h \to m_k^* \otimes Y_{kl^{-1}} \otimes m_l$ runs over all morphisms in $\hat{\B}$ satisfying (i)--(iv), and $f_{g,h}$ runs over morphisms $X_{gh^{-1}} \to Y_{gh^{-1}}$ in $\hat{\B}$. 
\end{lemma}

\begin{proof}
It is easy to check that the morphisms on the right hand side in the statement are indeed $A$-$A$-bimodule morphisms. 

We first show the inclusion ``$\subset$''.
Let $X\overset{\hat{f}}{\to}Y$ be a morphism in $\Hom_{\hat{\B}}(X,Y)$. Then
    \[J(\hat{f})=\bigoplus_{g,h \in G} \id_{m_g^*} \otimes \hat{f}_{gh^{-1}} \otimes \id_{m_h}=\left[\bigoplus_{g,h\,\in\,G} m_g^* \otimes X_{gh^{-1}} \otimes m_h\xrightarrow{\underset{g,h,k,l\,\in\,G}{\bigoplus} f_{g,h,k,l}}\bigoplus_{k,l\,\in\,G} m_k^* \otimes Y_{kl^{-1}} \otimes m_l\right]\]
    immediately implies that conditions $(i)$ -- $(iii)$ hold with $f_{g,h}\coloneqq \hat{f}_{gh^{-1}}$. Further, $f_{g,e}=\hat{f}_g=f_{gh,h}$ for all $g,h\in G$, and so $(iv)$ holds as well.

To show the inclusion ``$\supset$'', let
    \[f=\left[\bigoplus_{g,h\,\in\,G} m_g^* \otimes X_{gh^{-1}} \otimes m_h\xrightarrow{\underset{g,h,k,l\,\in\,G}{\bigoplus} f_{g,h,k,l}}\bigoplus_{k,l\,\in\,G} m_k^* \otimes Y_{kl^{-1}} \otimes m_l\right]\]
    satisfy conditions $(i)$ -- $(iv)$, i.e.
    \[f=\left[\bigoplus_{g,h\,\in\,G} m_g^* \otimes X_{gh^{-1}} \otimes m_h\xrightarrow{\underset{g,h\,\in\,G}{\bigoplus} \id_{m_g^*}\,\otimes\,f_{g,h}\,\otimes\,\id_{m_h}}\bigoplus_{g,h\,\in\,G} m_g^* \otimes Y_{gh^{-1}} \otimes m_h\right]\ , \] 
    and $f_{g,h}$ satisfies condition $(iv)$.
    Define $[X\overset{\hat{f}}{\to} Y]\in \Hom_{\hat{\B}}(X,Y)$ via $[X_g\overset{\hat{f}_g}{\to} Y_g]\coloneqq f_{g,e}$ for all $g\in G$. Then 
    \[\hat{f}_{gh^{-1}}=f_{gh^{-1},e}\overset{(iv)}{=}f_{gh^{-1}h,h}=f_{g,h}\]
    and hence $J(\hat{f})=f$.
\end{proof}

\subsection[The functor \texorpdfstring{$E$}{E}]{The functor \texorpdfstring{\boldmath $E$}{E}}\label{sec:functorE}

We will lift the functor $J$ to a functor $E$ such that the following diagram of functors commutes (on the nose):
\[
\begin{tikzcd}
    \hat{\B}^G \arrow[d,"\text{forget}"'] \arrow[r,"E"] & \B_\A \arrow[d,"\text{forget}"] \\
    \hat\B \arrow[r,"J"]  & {}_A \B_A 
\end{tikzcd}
\]
To construct $E$ it is useful to investigate how $J$ interacts with the actions on $T$. Hence, we take a closer look at the $A_0$-$(A_1\otimes A_2)$-bimodules obtained by tensoring $T$ and $J(X)$. Here, $A_{0,1,2}$ and also $B$ below all stand for $A$, the different notation is intended to show which copy of $A$ acts where:
\begin{align*}
    J(X) \otimes_0 T &= {\vphantom{\Big)}}_{A_0}\Bigl({}_{A_0}J(X)_B\otimes_B {}_BT_{A_1 A_2}\Bigr)_{A_1 A_2},\\
    T\otimes_1 J(X) &= {\vphantom{\Big)}}_{A_0}\Bigl({}_{A_0}T_{B A_2}\otimes_B {}_BJ(X)_{A_1}\Bigr)_{A_1 A_2},\\
    T\otimes_2 J(X) &= {\vphantom{\Big)}}_{A_0}\Bigl({}_{A_0}T_{A_1 B}\otimes_B {}_BJ(X)_{A_2}\Bigr)_{A_1 A_2}.
\end{align*}

We spell out the components, always taking the indices in advance in such a way that the outer $A_0$-, $A_1$- and $A_2$-action includes the projection to the action of $A_g \subset A_0$, $A_h \subset A_1$, and $A_k \subset A_2$. 
Tensoring with $J(X)$ from the left gives
\begin{align*}
J(X) \otimes_0 T &= \bigoplus_{g,h,k,l\,\in\,G} \Bigl[m_g^* \otimes X_{gl^{-1}} \otimes m_l\Bigr] \otimes_A \Bigl[m_{hk}^* \otimes m_h \otimes m_k\Bigr]
&\overset{(*)}= \bigoplus_{g,h,k\,\in\,G} m_g^* \otimes X_{g(hk)^{-1}} \otimes m_h \otimes m_k.
\end{align*}
In $(*)$ we used Lemma~\ref{lem:tensor-A} to compute the relative tensor product $\otimes_A$, which is non-zero only if $l = hk$.

Tensoring with $J(X)$ from the right over $A_1$ gives
\begin{align*}
T \otimes_1 J(X) &= \bigoplus_{g,h,k,l\,\in\,G} \Bigl(m_g^* \otimes m_{gk^{-1}} \otimes m_k\Bigr) \otimes_A \Bigl(m_l^* \otimes X_{lh^{-1}} \otimes m_h\Bigr)
\\
  &\overset{(1)}\cong \bigoplus_{g,h,k,l\,\in\,G} \Bigl(m_g^* \otimes \varphi\big( (gk^{-1} )^{-1}\big)(m_k)  \otimes m_{gk^{-1}} \Bigr) \otimes_A \Bigl(m_l^* \otimes X_{lh^{-1}} \otimes m_h\Bigr)
\\
&\overset{(2)}= \bigoplus_{g,h,k\,\in\,G} m_g^* \otimes 
\varphi\big( k g^{-1} \big)(m_k) 
\otimes X_{gk^{-1}h^{-1}} \otimes m_h.
\end{align*}
In step (1) we use the $G$-braiding $\tilde c_{m_{gk^{-1}},m_k}$ to exchange the two objects (the $G$-braiding $c_{m_{gk^{-1}},m_k}$ would not be compatible with the $A$-action). Step (2) is Lemma~\ref{lem:tensor-A}.

Note that here the external $A_0$ action is defined on $m_g^*$, the $A_2$ action is defined on $\varphi((gk^{-1})^{-1})(m_k)$, and the $A_1$ action is defined on $m_h$. We would prefer to explicitly have the order $A_0$-$A_1$-$A_2$, so we use the isomorphism
$$\phi_1 \colon T \otimes_1 J(X) \isomto \bigoplus_{g,h,k\,\in\,G} m_g^* \otimes X_{gk^{-1}h^{-1}} \otimes m_h \otimes m_k$$
given by 
\[
\phi_1 ~=~
\sum_{g,h,k \in G}
\includegraphics[scale=4.50,valign=c]{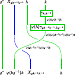}
\qquad .
\]
Tensoring with $J(X)$ from the right over $A_2$ results in
\begin{align*}
T \otimes_2 J(X) &= \bigoplus_{g,h,k,l\,\in\,G} \Bigl(m_g^* \otimes m_h \otimes m_{h^{-1}g}\Bigr) \otimes_A \Bigl(m_l^* \otimes X_{lk^{-1}} \otimes m_k\Bigr)
\\
&= \bigoplus_{g,h,k\,\in\,G} m_g^* \otimes m_h \otimes X_{h^{-1}gk^{-1}} \otimes m_k.
\end{align*}
Likewise, to reorder the tensor factors here, we use the isomorphism 
$$\phi_2\colon T \otimes_2 J(X) \isomto \bigoplus_{g,h,k\,\in\,G} m_g^* \otimes \varphi(h^{-1})\big(X_{h^{-1}gk^{-1}}\big) \otimes m_h \otimes m_k$$
given by 
\[
\phi_2 ~=~\sum_{g,h,k \in G}
\includegraphics[scale=4.50,valign=c]{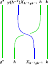}
\qquad .
\]
In total, we have
\begin{align*}
J(X) \otimes_0 T &= \bigoplus_{g,h,k\,\in\,G} m_g^* \otimes X_{gk^{-1}h^{-1}} \otimes m_h \otimes m_k,\\
T \otimes_1 J(X) \overset{\phi_1}&{\iso} \bigoplus_{g,h,k\,\in\,G} m_g^* \otimes X_{gk^{-1}h^{-1}} \otimes m_h \otimes m_k,\\
T \otimes_2 J(X) \overset{\phi_2}&{\iso} \bigoplus_{g,h,k\,\in\,G} m_g^* \otimes 
\varphi(h^{-1})\big(X_{h^{-1}gk^{-1}}\big)
 \otimes m_h \otimes m_k.
\end{align*}

The isomorphism $\phi_1$ and $\phi_2$ can also be used to give the morphisms $\tau_{1,2}$, $\overline\tau_{1,2}$ which form part of the data of an object in $\B_\A$: 
{\allowdisplaybreaks
    \begin{align*}
        \tau_1 &\coloneqq        
        \left[\underbrace{\bigoplus_{g,h,k\,\in\,G} m_g^* \otimes X_{gk^{-1}h^{-1}} \otimes m_h \otimes m_k}_{J(X)\,\otimes_0\,T} 
        \xrightarrow{\tau_1^{g,h,k}} T \otimes_2 J(X)\right], 
        \\
        \tau_2 &\coloneqq        
        \left[\underbrace{\bigoplus_{g,h,k\,\in\,G} m_g^* \otimes X_{gk^{-1}h^{-1}} \otimes m_h \otimes m_k}_{J(X)\,\otimes_0\,T} 
        \xrightarrow{\tau_2^{g,h,k}} T \otimes_2 J(X)\right], 
        \\
        \overline{\tau_1} &\coloneqq 
        \left[\underbrace{\bigoplus_{g,h,k\,\in\,G} m_g^* \otimes X_{gk^{-1}h^{-1}} \otimes m_h \otimes m_k}_{T \otimes_1 J(X)} 
        \xrightarrow{\overline{\tau}_1^{g,h,k}} 
        J(X) \otimes_0 T \right],
        \\
        \overline{\tau_2} &\coloneqq \left[\underbrace{\bigoplus_{g,h,k\,\in\,G} m_g^* \otimes \varphi\left(h^{-1}\right)\left(X_{h^{-1}gk^{-1}}\right) \otimes m_h \otimes m_k}_{T \otimes_2 J(X)} 
        \xrightarrow{\overline{\tau}_2^{g,h,k}} 
        J(X)\,\otimes_0\,T \right].
    \end{align*}}
where
    \begin{align*}
        \tau_1^{g,h,k} &= \dm{gk^{-1}} \Bigg(\frac{\dm{hk}}{\dm{h}}\Bigg)^{\!\frac{1}{2}} \phi_1^{-1} \ ,
        &
        \tau_2^{g,h,k} &=
        \dm{h^{-1}g} \Bigg(\frac{\dm{hk}}{\dm{h}}\Bigg)^{\!\frac{1}{2}}
        \phi_2^{-1}\,\circ\,\left[\bigoplus_{g,h,k\,\in\,G} \id_{m_g^*}\,\otimes\,\left(\eta^X_{h^{-1}}\right)^{-1}\,\otimes\,\id_{m_h\otimes m_k}\right] \ ,
        \\
        \overline{\tau}_1^{g,h,k} &=
        \dm{hk} \Bigg(\frac{\dm{h}}{\dm{hk}}\Bigg)^{\!\frac{1}{2}} \phi_1 \ ,
        &
        \overline{\tau}_2^{g,h,k} &= \dm{hk}
        \Bigg(\frac{\dm{k}}{\dm{hk}}\Bigg)^{\!\frac{1}{2}}
        \left[\bigoplus_{g,h,k\,\in\,G} \id_{m_g^*}\,\otimes\,\eta^X_{h^{-1}}\,\otimes\,\id_{m_h\otimes m_k}\right]\,\circ\,\phi_2 \ .
    \end{align*}
If we include explicitly the projection and embedding for the relative tensor products, as string diagrams, $\tau_1$ and $\tau_2$ are given by (recall from below Lemma~\ref{lem:tensor-A} that there is an extra dimension factor included in the embedding map):
{\allowdisplaybreaks
\begin{align*}    
\iota_{T,J(X)} \circ \tau_1 \circ \pi_{J(X),T} &~=~
\sum_{g,h,k}
\Bigg(\frac{\dm{hk}}{\dm{h}}\Bigg)^{\!\frac{1}{2}}
~~
\includegraphics[scale=4.25,valign=c]{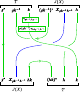}
~~, 
\\
\iota_{T,J(X)} \circ \tau_2 \circ \pi_{J(X),T} &~=~
\sum_{g,h,k}
\Bigg(\frac{\dm{hk}}{\dm{k}}\Bigg)^{\!\frac{1}{2}}
~~
\includegraphics[scale=4.25,valign=c]{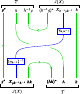}
\ .
\end{align*}}The morphism $\gamma_{h,h^{-1}}$ appears when computing the inverse of $\tilde c$ which is part of $\phi_2$, cf.\ the discussion of $G$-braidings in Section~\ref{sec:crossed_equiv}.
The string diagrams for $\overline\tau_1$ and $\overline\tau_2$ are similar.

The definition of the functor $E$ above still includes the coherence isomorphisms of the $G$-action. To simplify the proofs, from this point on we will make use of Remark~\ref{rem:G-strict} and assume the $G$-action to be strict.

\begin{proposition}
 The assignment
    \begin{align*}
        E \colon \hat{\B}^G &\to \B_\A, \\
        \left(X, \eta^X\right) &\mapsto \big(
    J(X),\tau_1,\tau_2,\overline{\tau_1},\overline{\tau_2}\big),\\
        \left[ \left(X, \eta^X\right) \overset{f}{\to} \left(Y, \eta^Y\right)\right] &\mapsto J(f).
    \end{align*}
indeed defines a functor $\hat{\B}^G \to \B_\A$.
\end{proposition}

\begin{proof}
    Fix objects $(X,\eta^X), (Y, \eta^Y)$ in  $\hat{\B}^G$ and a morphism $f \in \Hom_{\hat{\B}^G}((X,\eta^X), (Y, \eta^Y))$. To show that $E$ is well\--defined, it suffices to observe that the defining axioms \eqrefT{1}--\eqrefT{7} of $\B_\A$ are satisfied by $E(X, \eta^X)$ and that $J(f)$ commutes with $T$-crossings. We give the details for axiom \eqrefT{3} as well as axiom \eqrefT{6} for $i=2$; the others follow similarly. 

\begin{figure}[p!]
    \centering
\begin{align*}
&\frac{\delta}{\dm{h^{-1}g}} \hspace{-1em}
\includegraphics[scale=3.00,valign=c]{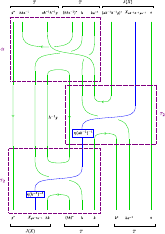}    
~=~\delta~
\includegraphics[scale=3.00,valign=c]{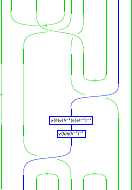}
\\
&=\delta~
\includegraphics[scale=3.00,valign=c]{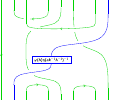}    
~=~\delta~
\includegraphics[scale=3.00,valign=c]{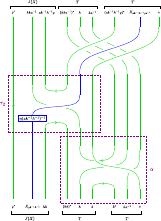}
\end{align*}
    \caption{Check that condition \eqrefT{3} holds. 
The factor $\delta$ is given by $\delta = (\dm{hk}/\dm{k})^{\frac12} (\dm{k}/\dm{a})^{\frac12} = (\dm{hk}/\dm{a})^{\frac12}$. 
}
    \label{fig:T3check}
\end{figure}

The calculation for \eqrefT{3} is given in Figure~\ref{fig:T3check}. The dimension factors arise from the $\psi_0^2$ in condition \eqrefT{3} and from the prefactor of $\tau_2$ (with projection/embedding).
The first equality there is a consequence of the naturality of the $G$-braiding of $X_{gk^{-1}h^{-1}}$ and $m_g$, the fact 
that we are now assuming the $G$-action to be strict (hence $\id= \varphi(h)\varphi(h^{-1})$), and the fact that the loop on $m_{h^{-1}g}$ can be replaced by the factor $\dim(m_{h^{-1}g})$. The second equality follows from the definition of $G$-equivariantisation and planar isotopy. The third equality results from the naturality of the $G$-braiding and the fact that the two diagrams are planarly isotopic. The remaining dimension factors agree with the prefactor of $\tau_2$ (with projection/embedding).

\begin{figure}[t!]
    \centering
$$\frac{1}{\dm{kg^{-1}hk}} \hspace{-2em}
\includegraphics[scale=3.50,valign=c]{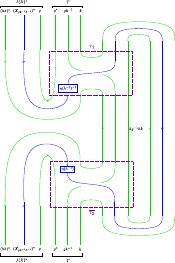}    
~~=~
\includegraphics[scale=3.50,valign=c]{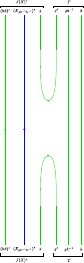}$$
    \caption{Check that condition \eqrefT{6} holds. The dimension factors from $\tau_2$ and $\overline\tau_2$ (with projection/embedding maps) cancel each other.}
    \label{fig:T6check}
\end{figure}

    Figure~\ref{fig:T6check} demonstrates that \eqrefT{6} holds for $i=2$.
    The equality is the result of planar isotopy, replacing the loop on $m_{kg^{-1}hk}$ with $\dim(m_{kg^{-1}hk})$, and the hexagon axiom for $G$-braidings.    
\end{proof}

\subsection[\texorpdfstring{$E$}{E} is a ribbon functor]{\texorpdfstring{\boldmath $E$}{E} is a ribbon functor}

We start by turning $E$ into a monoidal functor. Writing $\mathbf{X}=(X,\eta^X), \mathbf{Y}=(Y,\eta^Y)\in \hat{\B}^G$, we need to give isomorphisms (natural for $E^2$),
\[
E^0 : \1 \to E(\1) 
~~,\quad
E_{\mathbf{X},\mathbf{Y}}^2 \colon
E(\mathbf{X}) \otimes_A E(\mathbf{Y})
\to 
E\big( \mathbf{X} \otimes\mathbf{Y}\big)
\ .
\]
For $E^0$, first recall from Section~\ref{sec:orbdata} the general definition of the tensor unit in $\B_\A$, and that in Section~\ref{sec:functorJ} we already found $J(\1) = A$ with $J^0 = \id$. 
We make the same ansatz for $E$,
\[
    E^0 = \id \ .
\]
For $E^2$ we need to give a morphism from 
\[
E(\mathbf{X}) \otimes_A E(\mathbf{Y})
=\left(J(X)\otimes_A J(Y)
\,,\, \tau_1^{J(X)\otimes_A J(Y)}
\,,\, \tau_2^{J(X)\otimes_A J(Y)} \,,\, \overline{\tau_1}^{J(X)\otimes_A J(Y)}
\,,\, \overline{\tau_2}^{J(X)\otimes_A J(Y)}\right)
\]
to
\[
E\big( \mathbf{X} \otimes\mathbf{Y}\big) = \left(J(X\otimes Y), \tau_1^{J(X\otimes Y)}, \tau_2^{J(X\otimes Y)}, \overline{\tau_1}^{J(X\otimes Y)}, \overline{\tau_2}^{J(X\otimes Y)}\right) \ .
\]
For $J$ we had $J^2 = \id$, and the same turns out to be true here, that is, we make the ansatz
\[
E_{\mathbf{X},\mathbf{Y}}^2 = \id \ .
\]
If we include the projection to the relative tensor product, we get 
\[
E^2_{\mathbf{X},\mathbf{Y}} \circ \pi_{J(X),J(Y)} ~=
\sum_{g,h,a \in G}
\includegraphics[scale=4.50,valign=c]{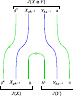}
\qquad .
\]
Note that if we also take the monoidal structure of the forgetful functor in the diagram in the beginning of Section~\ref{sec:functorE} to consist of identity morphisms, we obtain a commuting diagram of monoidal functors.

\begin{lemma}
    $E$ is a monoidal functor with $E^0$ and $E^2$ identities.
\end{lemma}

\begin{proof}
We start with $E^0$. 
The tensor unit in $\B_\A$ was given in Section~\ref{sec:orbdata} in terms of the $A$-(co)action on $T$ and $\omega$ as $\1_{\B_\A}:= (A, \tau_1^A, \tau_2^A, \overline{\tau_1}^A, \overline{\tau_2}^A)$. For our orbifold datum, the $\tau_i^A$ are explicitly given by (we also write out the projection/embedding maps)
\begin{align*}  
\iota_1 \circ \tau_1^A \circ \pi_0
&= \sum_{h,k \in G}
\frac{\sqrt{\dm{hk} \dm{h}}}{\dm{h}}
\includegraphics[scale=4,valign=c]{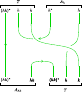}
~,
\\
\iota_2 \circ \tau_2^A \circ \pi_0
&=\sum_{h,k \in G}
\frac{\sqrt{\dm{hk} \dm{k}}}{\dm{k}}
\includegraphics[scale=4,valign=c]{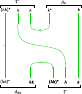} \ .
\end{align*}
Here, $\pi_0 : A \otimes T \to T = T \otimes_1 A$ is the action, and $\iota_i : T\otimes_i A = T \to T \otimes A$ the coaction.  
The dimension factors under the square root arise from $\omega_i$ (cf.\ Section~\ref{sec:orbdata}), and the $d_h$ in the denominator for $\tau_1$ is the normalisation of the coproduct of $A_h$ which enters the definition of the coaction (and similar for $\tau_2$).
We need to show that $E^0$ is indeed a morphism in $\B_\A$, and since $E^0=\id$ this amounts to checking that 
\[
\iota_i \circ \tau_i^A \circ \pi_0
~=~
\iota_{T,J(\1)} \circ \tau_i \circ \pi_{J(\1),T} \]
for $i=1,2$ and $\tau_i$ that for $J(\1)$. This is straightforward from the string diagrams given in Section~\ref{sec:functorE}.

As an aside we note that the fact that we can choose $E^0$ to be the identity map is the reason to include the a priori arbitrary looking square roots of dimension factors in the definition of $\tau_{1,2}$ in Section~\ref{sec:functorE}.

Next we check that $E^2=\id$ is a morphism in $\B_\A$, i.e.\ that $E^2$ commutes with $T$-crossings. The computations for $\tau_1$ and $\tau_2$ are similar, and we give the one for $\tau_2$ explicitly in Figure~\ref{fig:tau2_E2}. The first equality uses the twisted hexagon identity (cf.\ Section~\ref{sec:crossed_equiv}).
The factor $1/\dm{h^{-1}a}$ in the last line in Figure~\ref{fig:tau2_E2} compensates the value of the additional loop. This agrees with the overall factor needed on the right hand side: $(d_{hk}/d_k)^{\frac12}$ and $(d_{a}/d_{h^{-1}a})^{\frac12}$ from the two $\tau_2$'s, and $(d_{a} d_{h^{-1}a})^{-\frac12}$ from the insertion of $\omega_2$ (cf.\ Section~\ref{sec:orbdata}).

\begin{figure}[p!]
\begin{align*}
&    \Bigg(\frac{\dm{hk}}{\dm{k}}\Bigg)^{\!\frac{1}{2}}
    \hspace{-1.5em}
    \includegraphics[scale=3.00,valign=c]{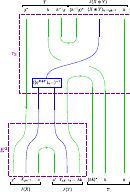}        
    = 
    \Bigg(\frac{\dm{hk}}{\dm{k}}\Bigg)^{\!\frac{1}{2}}
    \hspace{-0.5em}
    \includegraphics[scale=3.00,valign=c]{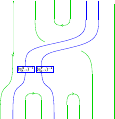}
\\    
&    = \frac{1}{\dm{h^{-1}a}}
    \Bigg(\frac{\dm{hk}}{\dm{k}}\Bigg)^{\!\frac{1}{2}}
    \includegraphics[scale=3.00,valign=c]{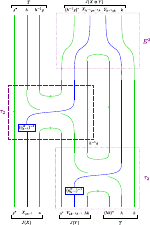}  
\end{align*}    
    \caption{The isomorphism $\tau_2$ commutes with the monoidal structure $E^2$ of $E$.
    }
    \label{fig:tau2_E2}
\end{figure}

The monoidal structure $E^2$ automatically satisfies the hexagon identity as on the underlying object in ${}_A\B_A$ it reduces to $J^2$, which does satisfy the hexagon.
\end{proof}

\begin{figure}[p!]
$$
\frac{\delta}{|G|}~
    \hspace{-1em}
    \includegraphics[scale=3.0,valign=c]{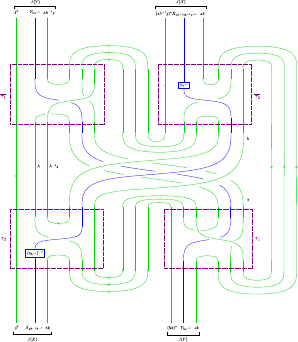}        
$$    
\caption{Substituting the expressions for $\tau_i$ and $\overline\tau_i$ into the braiding of $\B_\A$. The factor $\delta$ is given by $\delta = (\dm{a h^{-1} g}\,\dm{h^{-1}\,g}\,\dm{k}\,\dm{h}\,\dm{a})^{-1}$ as explained in the text.
}
    \label{fig:braid}
\end{figure}

Next we address compatibility of $E$ with the braiding. Inserting the definitions of $\tau_1, \tau_2, \overline{\tau_1}$ and $\overline{\tau_2}$ into the definition of the braiding $c^{\B_\A}$ from Section~\ref{sec:orbdata} results in the diagram shown in Figure~\ref{fig:braid}. 
More precisely, a homogeneous component of $c^{\B_\A}_{E(\mathbf{X},\mathbf{Y})}$ is shown (including the embedding and projection maps).
The factor $\delta$ is given by 
\begin{align*}   
    \delta &= 
    \Bigg(\frac{\dm{ak}}{\dm{a}}\Bigg)^{\!\frac{1}{2}} 
    \Bigg(\frac{\dm{hk}}{\dm{k}}\Bigg)^{\!\frac{1}{2}} 
    \Bigg(\frac{\dm{a}}{\dm{ah^{-1}g}}\Bigg)^{\!\frac{1}{2}} 
    \Bigg(\frac{\dm{k}}{\dm{ak}}\Bigg)^{\!\frac{1}{2}} 
    \frac{1}{\dm{h^{-1}g}\dm{k}} \,
    \frac{1}{\dm{h}\dm{a}}
    \Bigg(\frac{1}{\dm{a h^{-1} g}}\Bigg)^{\!\frac{1}{2}} 
    \Bigg(\frac{1}{\dm{hk}}\Bigg)^{\!\frac{1}{2}} 
    \\
    &= \frac{1}{\dm{a h^{-1} g}\,\dm{h^{-1}\,g}\,\dm{k}\,\dm{h}\,\dm{a}}
    \ ,
\end{align*}
where the various dimension factors arise, in this order, from: the factors for $\tau_1$, $\tau_2$, $\overline{\tau_1}$,  $\overline{\tau_2}$ (all with their embedding/projection maps), the $\omega^2$ below and above the braiding morphism, the $\psi^0$ below and above the braiding morphism. 
This is simplified in Figure~\ref{fig:braidsimp}. Reindexing the objects in last diagram in Figure~\ref{fig:braidsimp} 
makes the diagram independent of $k$, 
    and summing over $k$ gives an additional factor of $|G|$. Overall, we get:
\[
\iota_{J(Y),J(X)} \circ c^{\B_\A}_{E(\mathbf{X}),E(\mathbf{Y})} \circ \pi_{J(X),J(Y)} = 
\sum_{g,h,k \in G}
\frac{1}{\dm{ah^{-1}g}}
    \includegraphics[scale=4.50,valign=c]{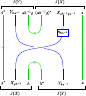}        
\]

\begin{figure}[t!]
\begin{align*}
&   \frac{1}{|G|} \frac{1}{\dm{ah^{-1}g}}
    \includegraphics[scale=4.00,valign=c]{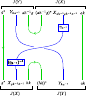}        
    = 
    \frac{1}{|G|} \frac{1}{\dm{ah^{-1}g}}
    \includegraphics[scale=4.00,valign=c]{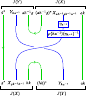}        
\\    
&    = 
    \frac{1}{|G|} \frac{1}{\dm{ah^{-1}g}}
    \includegraphics[scale=4.00,valign=c]{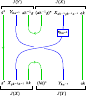}
\end{align*}    
    \caption{Simplifying the expression for the braiding given in Figure~\ref{fig:braid}.}
    \label{fig:braidsimp}
\end{figure}

\begin{lemma}
    $E$ is a braided functor.
\end{lemma}

\begin{proof}
This statement is equivalent to the claim that the diagram
\[\begin{tikzcd}
	{E(\mathbf{X})\otimes_A E(\mathbf{Y})} && {E(\mathbf{Y})\otimes_A E(\mathbf{X})} \\
	{E(\mathbf{X}\otimes \mathbf{Y})} && {E(\mathbf{Y}\otimes \mathbf{X})}
	\arrow["{{c^{\B_\A}_{E(\mathbf{X}),E(\mathbf{Y})}}}", from=1-1, to=1-3]
	\arrow["{E^2_{\mathbf{X},\mathbf{Y}}}"', from=1-1, to=2-1]
	\arrow["{E^2_{\mathbf{Y},\mathbf{X}}}", from=1-3, to=2-3]
	\arrow["{E(c^{\hat{\B}^G})}"', from=2-1, to=2-3]
\end{tikzcd}\]
commutes for all $\mathbf{X}=(X,\eta^X), \mathbf{Y}=(Y,\eta^Y)\in \hat{\B}^G$.
This can be checked in terms of string diagrams: 
\[
\frac{1}{\dm{ah^{-1}g}}
\hspace{-2em}
    \includegraphics[scale=3.50,valign=c]{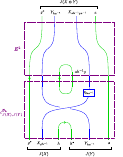}        
    =
    \includegraphics[scale=3.50,valign=c]{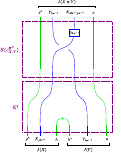}     
    \raisebox{-10em}{\qedhere}
\]
\end{proof}

\begin{figure}[t]
\[
\frac{1}{\dm{g^{-1}a}} 
    \includegraphics[scale=3.50,valign=c]{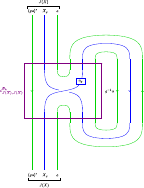}        
\]
\caption{The twist of $\B_\A$, evaluated on $E(\mathbf{X})$ for $\mathbf{X} = (X,\eta^X)$.}
    \label{fig:twist}
\end{figure}

The twist of $\B_\A$, evaluated on $E(\mathbf{X})$, is given in Figure~\ref{fig:twist} and can easily be simplified to
\[
\theta^{\B_\A}_{E(\mathbf{X})} = 
    \includegraphics[scale=4.00,valign=c]{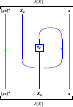}        
\]

\begin{lemma}
    $E$ is a ribbon functor.
\end{lemma}

\begin{proof}
We need to check that the twist satisfies $$\theta^{\B_\A}_{E(X,\eta^X)}=E\left(\theta^{\hat{\B}^G}_{(X,\eta^X)}\right).$$ 
The left hand side was already computed above. Using the expression for ribbon twist in the equivariantised category from Section~\ref{sec:crossed_equiv}, we see that the right hand side is
\[
E\left(\theta^{\hat{\B}^G}_{(X,\eta^X)}\right)
=\includegraphics[scale=4.00,valign=c]{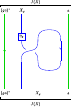} \ ,
\]
which is indeed equal to $\theta^{\B_\A}_{E(X,\eta^X)}$ by naturality of the $G$-braiding.
\end{proof}

\subsection[\texorpdfstring{$E$}{E} is an equivalence]{\texorpdfstring{\boldmath $E$}{E} is an equivalence}

To complete the proof of Theorem~\ref{thm:main}, we still need to show that the functor $E$ is an equivalence. This will be done in the next two lemmas.

Recall that $m_g$ is a $\1$-$A_g$-bimodule, and $m_g^*$ an $A_g$-$\1$-bimodule. These two bimodules provide a Morita-context: $A_g = m_g^* \otimes m_g$ and $m_g \otimes_A m_g^* \cong \1$ (cf.\ Lemma~\ref{lem:tensor-A}).

\begin{lemma}\label{lem:AgAh-bimodhom}
\begin{enumerate}
	\item Let $M$ be an $A$-$A$-bimodule in $\B \equiv \B_e$. Then there exist $N_{g,h} \in \B_{gh^{-1}}$ s.t.
	\[
	M \cong \bigoplus_{g,h \in G} m_g^* \otimes N_{g,h} \otimes m_h\]
	as $A$-$A$-bimodules.

	\item Let now $M =  \bigoplus_{g,h \in G} m_g^* \otimes N_{g,h} \otimes m_h$, $M' =  \bigoplus_{g,h \in G} m_g^* \otimes N'_{g,h} \otimes m_h$ and $F : M \to M'$ an $A$-$A$-bimodule map. Then there are $f_{g,h} : N_{g,h} \to N'_{g,h}$ such that
	\[
	F = \sum_{g,h} \id_{m_g^*} \otimes f_{g,h} \otimes \id_{m_h} \ .
	\]
\end{enumerate}
\end{lemma}

\begin{proof}
Part 1:
The units $\eta_g :  \1 \to A_g$ of the individual algebras act as idempotents. Set
    $$M_{g,h} := \text{im}(\eta_g \otimes_A M \otimes_A \eta_h) $$
Then $M = \bigoplus_{g,h \in G} M_{g,h}$ as $A$-$A$-bimodules. We have the following isomorphism of $A_g$-$A_h$-bimodules,
\[
M_{g,h} \cong A_g \otimes_A M \otimes_A A_h \cong m_g \otimes m_g^* \otimes_A M \otimes_A m_h \otimes m_h^* \ ,
\]
so that we can take $N_{g,h} := m_g^* \otimes_A M \otimes_A m_h$.

\medskip

\noindent
Part 2: Since $F$ commutes with the action of the units $\eta_g$ on both sides, it decomposes as $F = \sum_{g,h} F_{g,h}$ with $F_{g,h} : m_g^* \otimes N_{g,h} \otimes m_h \to  m_g^* \otimes N'_{g,h} \otimes m_h$ an $A_g$-$A_h$-bimodule map. Since $m_g$ and $m_h$ are Morita-contexts, the functor $m_g^* \otimes (-) \otimes m_h$ is an equivalence of categories from $\B_{gh^{-1}}$ to $A_g$-$A_h$-bimodules in $\B$, and hence $F_{g,h} = \id_{m_g^*} \otimes f_{g,h} \otimes \id_{m_h}$ for appropriate $f_{g,h}$.
\end{proof}

\begin{lemma}
    $E$ is an equivalence of categories.
\end{lemma}

\begin{proof}
We will show that $E$ is fully faithful and essentially surjective.

\medskip

\noindent
\textit{Faithful:} This is clear as $J$ is faithful and morphism spaces in $\B_\A$ are subspaces of bimodule maps.

\medskip

\noindent
\textit{Full:}
Let $F\colon E\left(X, \eta^X\right) \to E\left(Y,\eta^Y\right)$ be a morphism in $\B_\A$.
This means, that
$F\colon J(X)\to J(Y)$ is an $A$-$A$-bimodule map such that the diagram
	\[\begin{tikzcd}
		J(X)\otimes_0 T \arrow[d, "\tau_i"] \arrow[r, "{F\,\otimes_0\,\id}"] & J(Y)\otimes_0 T \arrow[d, "\tau_i"] \\
		T\otimes_i J(X) \arrow[r, "{\id\,\otimes_i\,F}"] & T\otimes_i J(Y)
	\end{tikzcd}\]
commutes for $i\in\{1,2\}$.
By Lemma~\ref{lem:AgAh-bimodhom}, we can write
$F = \sum_{g,h} \id_{m_g^*} \otimes f_{g,h} \otimes \id_{m_h}$ for some $f_{g,h} : X_{gh^{-1}} \to Y_{gh^{-1}}$.
The above commuting diagram, together with the definition of $\tau_i$ for $E\left(X, \eta^X\right)$ and $E\left(Y,\eta^Y\right)$ results in the conditions
\begin{align*}
\tau_1 &: f_{g,hk} = f_{gk^{-1},h} \ ,
\\
\tau_2 &: (\eta_{h^{-1}}^Y)^{-1} \circ f_{g,hk} = \varphi(h^{-1})f_{gk^{-1},h} \circ (\eta_{h^{-1}}^X)^{-1} \ .
\end{align*}
The condition for $\tau_1$ shows that $f_{g,k} = f_{gk^{-1},e}$, and so by Lemma~\ref{lem:image-J-Hom}, $F = J(\hat f)$, where $\hat f : X \to Y$ has components $\hat f_g = f_{g,e}$. Condition $\tau_2$ ensures that $\hat f : (X,\eta^X) \to (Y,\eta^Y)$ is a morphism in $\hat\B^G$. Indeed, setting $h = k^{-1}$ results in $f_{g,e} \circ \eta_{k}^X = \eta_k^Y \circ \varphi(k)f_{g,e}$.

\medskip

\noindent
\textit{Essentially surjective:}
Let $(M, \tau_1, \tau_2, \overline{\tau_1}, \overline{\tau_2})) \in \B_\A$.
We need to give $(X,\eta^X) \in \hat{\B}^G$ such that $(M, \tau_1, \tau_2, \overline{\tau_1}, \overline{\tau_2}) \cong E(X,\eta^X)$. We first construct $X \in \hat\B$ such that $M \cong J(X)$ as bimodules, and then find the equivariant structure $\eta^X$ on $X$.

First note that by Lemma~\ref{lem:AgAh-bimodhom} we may assume that $M = \bigoplus_{a,b \in G} m_a^* \otimes N_{a,b} \otimes m_b$ for some $N_{a,b} \in \B_{ab^{-1}}$.
We now proceed along the same lines as when defining the functor $E$ in Section~\ref{sec:functorE}, except that here we do not already know that $M \cong J(X)$ for some $X$.
Namely, we have
\begin{align*}
	M \otimes_0 T &= \bigoplus_{g,h,k,x \in G} (m_g^* \otimes N_{g,x} \otimes m_x) \otimes_A (m_{hk}^* \otimes m_h \otimes m_k)
\overset{\text{Lem.\,\ref{lem:tensor-A}}}=
\bigoplus_{g,h,k \in G} m_g^* \otimes N_{g,hk} \otimes m_h \otimes m_k \ ,
\\
T \otimes_1 M &\cong \bigoplus_{g,h,k\,\in\,G} m_g^* \otimes
\varphi\big( k g^{-1} \big)(m_k)
\otimes N_{gk^{-1},h} \otimes m_h \ ,
\\
T \otimes_2 M &= \bigoplus_{g,h,k\,\in\,G} m_g^* \otimes m_h \otimes N_{h^{-1}g,k} \otimes m_k \ ,
    \end{align*}
where the outer $A_0$-, $A_1$- and $A_2$-action includes the projection to the action of $A_g \subset A_0$, $A_h \subset A_1$, and $A_k \subset A_2$.
Define the $A_0$-$A_1\otimes A_2$-bimodule isomorphism $\varphi_1: T \otimes_1 M \isomto \bigoplus_{g,h,k\,\in\,G} m_g^* \otimes N_{gk^{-1},h} \otimes m_h \otimes m_k$ and $\varphi_2 : T \otimes_2 M \isomto \bigoplus_{g,h,k\,\in\,G} m_g^* \otimes \varphi(h^{-1})\big(N_{h^{-1}g,k}\big) \otimes m_h \otimes m_k$ by
\[
\varphi_1 =
\includegraphics[scale=4.50,valign=c]{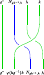} ~~,\quad
\varphi_2 =
\includegraphics[scale=4.50,valign=c]{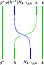} \ .
\]
Then also
\begin{align*}
\varphi_1 \circ \tau_1 &: \bigoplus_{g,h,k \in G} m_g^* \otimes N_{g,hk} \otimes m_h \otimes m_k \isomto \bigoplus_{g,h,k \in G} m_g^* \otimes N_{gk^{-1},h} \otimes m_h \otimes m_k
~,
\\
\varphi_2 \circ \tau_2&: \bigoplus_{g,h,k \in G} m_g^* \otimes N_{g,hk} \otimes m_h \otimes m_k \isomto \bigoplus_{g,h,k \in G} m_g^* \otimes \varphi(h^{-1}) N_{h^{-1}g,k} \otimes m_h \otimes m_k
\end{align*}
are $A_0$-$A_1\otimes A_2$-bimodule isomorphisms.
As in the proof of Lemma~\ref{lem:AgAh-bimodhom}\,(ii) we conclude that
\begin{align*}
\varphi_1 \circ \tau_1 &= \bigoplus_{g,h,k \in G}\dm{gk^{-1}} \Bigg(\frac{\dm{hk}}{\dm{h}}\Bigg)^{\!\frac{1}{2}} \cdot  \id_{m_g^*} \otimes t^1_{g,h,k} \otimes \id_{m_h} \otimes \id_{m_k}
\\
\varphi_2 \circ \tau_2 &= \bigoplus_{g,h,k \in G}
\dm{h^{-1}g} \Bigg(\frac{\dm{hk}}{\dm{h}}\Bigg)^{\!\frac{1}{2}}
\cdot
 \id_{m_g^*} \otimes t^2_{g,h,k} \otimes \id_{m_h} \otimes \id_{m_k}
\end{align*}
for some isomorphisms
\[
    t^1_{g,h,k}: N_{g,hk} \isomto N_{gk^{-1}, h}
    ~~,\quad
    t^2_{g,h,k}: N_{g,hk} \isomto \varphi(h^{-1})N_{h^{-1}g, k}
    \ .
\]
in $\B_{gk^{-1}h^{-1}}$. 
The dimension factors mirror those for $\tau_i^{g,h,k}$ in Section~\ref{sec:functorE}.
Inverting this relation to extract $\tau_i$, we get
{\allowdisplaybreaks
\begin{align*}    
\iota_{T,M} \circ \tau_1 \circ \pi_{M,T} &~=~
\Bigg(\frac{\dm{hk}}{\dm{h}}\Bigg)^{\!\frac{1}{2}}
~~
\includegraphics[scale=4.25,valign=c]{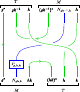}
~~, 
\\
\iota_{T,M} \circ \tau_2 \circ \pi_{M,T} &~=~
\Bigg(\frac{\dm{hk}}{\dm{k}}\Bigg)^{\!\frac{1}{2}}
~~
\includegraphics[scale=4.25,valign=c]{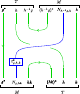}
\ .
\end{align*}}
For later use, we record the following conditions on the $t^i_{g,h,k}$ implied by the fact that $\tau_1$, $\tau_2$ satisfy \eqrefT{1}--\eqrefT{3}. Namely, for all $a,g,h,k \in G$ we have: 
\begin{align*}
	t^1_{g,h,k} &= t^1_{ga^{-1}, h, ka^{-1}} \circ t^1_{g, hka^{-1},a}
	&& \eqrefT{1}
	\\
	\varphi(h^{-1})t^1_{h^{-1}g, ka^{-1}, a} \circ t^2_{g,h,k} &= t^2_{ga^{-1},h,ka^{-1}} \circ t^1_{g,hka^{-1},a}
	&& \eqrefT{2}
	\\
	\varphi(h^{-1})t^2_{h^{-1}g, ka^{-1}, a} \circ t^2_{g,h,k} &= t^2_{g,hka^{-1},a}
	&& \eqrefT{3}
\end{align*}
To check e.g.\ the last condition, one can use Figure~\ref{fig:T3check} as a guide and replace the $\eta$-morphisms there by $t^2$ morphisms as in the string diagram for $\tau_2$ above.

We can now define $X = \bigoplus_{g \in G} X_g \in \hat\B$ such that $M \cong J(X)$ as $A$-$A$-bimodules. Namely, we set
\[
X_g := N_{g,e}
\]
and observe that $t^1_{g,h,k}$ specialised to $h=e$ gives an isomorphism $t^1_{g,e,k}: N_{g,k} \isomto N_{gk^{-1}, e} = X_{gk^{-1}}$.

Next we turn to the equivariant structure $\eta^X_a : \varphi(a)X \isomto  X$. We define it  componentwise via
\[
\eta^X_a\big|_{\varphi(a)X_b}:= 
\Big[ \varphi(a)X_{b} = \varphi(a)N_{b,e} \xrightarrow[]{\left(t^2_{a^{-1}b,a^{-1},e}\right)^{-1}} N_{a^{-1}b,a^{-1}} \xrightarrow[]{\left(t^1_{a^{-1}ba, a^{-1}, a}\right)^{-1}} N_{a^{-1}ba,e} = X_{a^{-1}ba} \Big] \ .
\]
The coherence condition condition for the $(\eta_a^X)_{a \in G}$ then amounts to commutativity of the outer diagram in:
\[
\begin{tikzcd}
	{\vphi(g) \vphi(h) N_{b,e}} &&& {\vphi(g)N_{h^{-1}b,h^{-1}}} &&& {\vphi(g)N_{h^{-1}bh,e}} \\
	\\
	&&&&&& {N_{g^{-1}h^{-1}bh,g^{-1}}} \\
	\\
	{\vphi(g) \vphi(h) N_{b,e}} &&& {N_{g^{-1}h^{-1}b,g^{-1}h^{-1}}} &&& {N_{g^{-1}h^{-1}bhg,e}}
	\arrow["{\vphi(g) \big(t^2_{h^{-1}b,h^{-1},e}\big)^{-1}}", from=1-1, to=1-4]
	\arrow["{\gamma_{g,h}=\id}"', from=1-1, to=5-1, equal]
	\arrow["\eqrefT{3}"{description}, draw=none, from=1-1, to=5-4]
	\arrow["{\vphi(g) \big(t^1_{h^{-1}bh,h^{-1},h}\big)^{-1}}", from=1-4, to=1-7]
	\arrow["\eqrefT{2}"{description}, draw=none, from=1-4, to=3-7]
	\arrow["{\big(t^2_{g^{-1}h^{-1}b,g^{-1},h^{-1}}\big)^{-1}}"{description}, from=1-4, to=5-4]
	\arrow["{\big(t^2_{g^{-1}h^{-1}bh,g^{-1},e}\big)^{-1}}", from=1-7, to=3-7]
	\arrow[""{name=0, anchor=center, inner sep=0}, "{\big(t^1_{g^{-1}h^{-1}bhg,g^{-1},g}\big)^{-1}}", from=3-7, to=5-7]
	\arrow["{\big(t^2_{g^{-1}h^{-1}b,g^{-1}h^{-1},e}\big)^{-1}}"', from=5-1, to=5-4]
	\arrow["{\big(t^1_{g^{-1}h^{-1}bh,g^{-1}h^{-1},h}\big)^{-1}}"{description}, from=5-4, to=3-7]
	\arrow[""{name=1, anchor=center, inner sep=0}, "{\big(t^1_{g^{-1}h^{-1}bhg,g^{-1}h^{-1},hg}\big)^{-1}}"', from=5-4, to=5-7]
	\arrow["\eqrefT{1}"', shift left=5, draw=none, from=1, to=0]
\end{tikzcd}
\]
Here we label the center of the subdiagrams by which coherence axioms imply their commutativity.

Having constructed our candidate $(X,\eta^X)$, we still need to check that it is isomorphic to 
$(M, \tau_1, \tau_2, \overline{\tau_1}, \overline{\tau_2})$. Consider the isomorphism of $A$-$A$-bimodules
\[
\Phi := \Big[ M = \bigoplus_{g,h \in G} m_g^* \otimes N_{g,h} \otimes m_h
\xrightarrow{\bigoplus_{g,h \in G} \id \otimes t^1_{g,e,h} \otimes \id} \bigoplus_{g,h \in G} m_g^* \otimes N_{gh^{-1},e} \otimes m_h = J(X) \Big] \ .
\]
We will check that $\Phi : (M, \tau_1, \tau_2, \overline{\tau_1}, \overline{\tau_2}) \to E(X,\eta^X)$ is an isomorphism in $\B_\A$. For this it remains to see that $\Phi$ commutes with $\tau_i$ as in Definition~\ref{def:cat-of-wilson}. One finds the conditions:
\begin{align*}	
\tau_1 ~&:~ t^1_{g,e,hk} = t^1_{gk^{-1},e,h} \circ t^1_{g,h,k}
\ ,
\\
\tau_2 ~&:~ \Big(\eta_{h^{-1}}\big|_{\varphi(h^{-1})X_{h^{-1}g k^{-1}}}\Big)^{-1} 
\circ
t^1_{g,e,hk}
= \varphi(h^{-1})t^1_{h^{-1}g,e,k} \circ t^2_{g,h,k} 
\ .
\end{align*}
The condition for $\tau_1$ is just a special case of the identity that $t^1$ satisfies because of \eqrefT{1} as listed above (set $a=k$, $h=e$ there). To verify the condition for $\tau_2$ we substitute the definition of $\eta$ to get
\[
t^2_{g k^{-1} ,h,e} \circ \underbrace{t^1_{g k^{-1} h^{-1}, h, h^{-1}} 
\circ 
t^1_{g,e,hk}}_{=t^1_{g,h,k} \text{ by \eqrefT{1}}} 
= 
\varphi(h^{-1})t^1_{h^{-1}g,e,k} \circ t^2_{g,h,k} \ ,
\]
which in turn holds by \eqrefT{2} (set $a=k$ there).
\end{proof}

This completes the proof of Theorem~\ref{thm:main}.

\newcommand\arxiv[2]      {\href{http://arXiv.org/abs/#1}{#2}}
\newcommand\doi[2]        {\href{http://dx.doi.org/#1}{#2}}

\end{document}